\documentclass[11pt]{scrartcl}
\usepackage[english]{babel}
\usepackage[T1]{fontenc}
\usepackage[utf8]{inputenc}
\usepackage{color,amssymb,amsfonts,amsthm,mathrsfs,graphicx,setspace,amsmath}
\usepackage{caption}
\usepackage[left=3cm,right=3cm,top=2.5cm,bottom=3cm]{geometry}
\usepackage{mathtools}
\usepackage{thm-restate}
\usepackage{xcolor}
\usepackage{tikz}
\usetikzlibrary{shapes.geometric,hobby,calc}
\usetikzlibrary{decorations.pathmorphing}
\usetikzlibrary{decorations.text}
\usetikzlibrary{shapes.misc}
\usetikzlibrary{calc}
\usetikzlibrary{decorations,shapes}
\usepackage{hyperref}
\usepackage{nameref,cleveref}
\usepackage[square, numbers, comma, sort&compress]{natbib}
\usepackage{enumitem}

\usepackage{macros}
\usepackage{dsfont}
\usepackage{soul}
\usepackage{xspace}
\usepackage[prependcaption,colorinlistoftodos]{todonotes}
\usepackage{titling}

\usepackage{import}
\usepackage[font=small]{caption}
\usepackage{subcaption}
\usepackage{multirow}
\usepackage{booktabs}
\usepackage{rotating}
\usepackage{array}
\newcolumntype{P}[1]{>{\centering\arraybackslash}p{#1}}
\newcolumntype{M}[1]{>{\centering\arraybackslash}m{#1}}

\usepackage{authblk}

\newcommand*\samethanks[1][\value{footnote}]{\footnotemark[#1]}

\hypersetup{
	colorlinks=true,
	linkcolor=myBlue,
	citecolor=myBlue,
	urlcolor=myLightBlue,
	bookmarksopen=true,
	bookmarksnumbered,
	bookmarksopenlevel=2,
	bookmarksdepth=3
}

\usetikzlibrary[positioning]
\usetikzlibrary{patterns}
\usetikzlibrary{shapes}
\usetikzlibrary{decorations.markings}
\usetikzlibrary{arrows}

\title{Directed \treewidth is closed under taking butterfly minors\thanks{This work is based on Gunwoo Kim’s bachelor’s thesis at the Technische Universität Berlin.}}

\author[1,2]{Gunwoo Kim\thanks{Supported by the Institute for Basic Science (IBS-R029-C1).}}
\author[2]{Meike Hatzel\samethanks}
\author[3]{Stephan Kreutzer}

\affil[1]{School of Computing, KAIST, Daejeon, South Korea}
\affil[2]{Discrete Mathematics Group, Institute for Basic Science (IBS), Daejeon, South Korea}
\affil[3]{Logic and Semantics Group, Technische Universität Berlin, Germany}
\affil[ ]{E-mail: \href{mailto:gunwoo.kim@kaist.ac.kr}{gunwoo.kim@kaist.ac.kr}, \href{mailto:research@meikehatzel.com}{research@meikehatzel.com}, \href{mailto:stephan.kreutzer@tu-berlin.de}{stephan.kreutzer@tu-berlin.de}}
\date{}

\newif\ifcomment
\commentfalse

\begin{document}

\maketitle

\vspace{-1cm}
\begin{abstract}
    Butterfly minors are a generalisation of the minor containment relation for undirected graphs to directed graphs.
    Many results in directed structural graph theory use this notion as a central tool next to directed treewidth, a generalisation of the width measure treewidth to directed graphs.
    Adler [JCTB'07] showed that the directed treewidth is not closed under taking butterfly minors.
    Over the years, many alternative definitions for directed treewidth appeared throughout the literature, equivalent to the original definition up to small functions.
    In this paper, we consider the major ones and show that not all of them share the problem identified by Adler.
\end{abstract}

\section{Introduction}

The width measure \emph{directed \treewidth} was first introduced by Reed~\cite{reed_introducing_1999} and Johnson, Robertson, Seymour and Thomas~\cite{johnson_directed_2001,johnson_addendum_2001}.
It soon became a central tool for the investigation of the structure of directed graphs \cite{2014pregridthm,kawarabayashi_directed_2015,kawarabayashi2022directed,campos2022fpt,2020dirflatwallI,2021dirflatwallII,2022dirflatwallIII} similar to the role \emph{\treewidth} plays in the context of the structure of graph minors by Robertson and Seymour~\cite{robertson_seymour_graphs}.

However, throughout the literature on directed \treewidth~\cite{reed_introducing_1999,johnson_directed_2001,johnson_addendum_2001,kreutzer_width-measures_2014,giannopoulou_directed_2022,kawarabayashi2022directed} one can find several different definitions for the concept.
This often has technical reasons in the sense that some definitions are more convenient for specific applications than others, and it was widely accepted that these definitions, which all lie within small functions of each other, could be interchanged for convenience.

Then, however, Adler~\cite{adler_directed_2007} discovered a major flaw of the original definition: The definition is not closed under taking \emph{butterfly minors}, which is a containment relation generalising minors for undirected graphs and used in all the aforementioned investigations of the structure of directed graphs.
This was often used as a discouraging argument against directed \treewidth as a tool \cite{hlineny2008width,wiederrecht2020dtwone,2019cyclewidth}.
Looking at the provided examples, however, shows that the gap is never really large.
Thus, with different definitions around that are similar but not quite the same, it raises the question of whether one of these definitions could be closed under taking butterfly minors.

The width measure \emph{cyclewidth}~\cite{2019cyclewidth}, introduced by Hatzel, Rabinovich and Wiederrecht, shows that there is at least one equivalent definition that is closed under taking butterfly minors.
However, the definition differs enough for it to carry a different name.
But it is equivalent up to a linear function as recently shown by Bowler, Ghorbani, Gut, Jacobs, and Reich~\cite{2024linearcyclewidthbound}.
Another hint as to why one might suspect one of the definitions to be closed is that the maximum size of a bramble in a digraph is closed under taking butterfly minors, which we show in~\cref{sec:brambles_closed}.

Indeed, it is true that not all the alternative definitions share the flaw of possibly growing when taking butterfly minors; rather, it turns out to be a technical result of the initial choice of definition.
The knowledge about there being definitions that are closed under taking butterfly minors has circulated within the community for a while.
In this paper, we conduct a thorough investigation of the major definitions and formally establish their properties regarding taking butterfly minors.

In~\cref{sec:dtw_definitions}, we state five definitions that can be found throughout the literature and conduct an extensive comparison between them.
This leads to identifying one definition that is closed under taking butterfly minors while the others are not.
For this reason, this definition is becoming the main definition used in the most recent literature~\cite{kawarabayashi2022directed,elementarypaper2024}.
We provide a more precise overview of our results in~\cref{sec:result_overview} after introducing the relevant definitions.

Throughout the paper, we denote the directed \treewidth of a digraph $D$ by $\dtw{D}$ in all statements that hold for all of the notions of directed \treewidth that we provide definitions of in~\cref{sec:dtw_definitions}.

\section{Preliminaries}

All graphs in this paper are directed, simple, and finite.
For a digraph $D$, we refer to its vertex set by $\V{D}$ and its edge set by $\E{D}$.
The \emph{out-neighbourhood} of a vertex $u$ in $D$ is defined by $\OutN{D}{u} \coloneqq \Set{v \in \V{D} \mid (u,v)\in\E{D}}$, and the \emph{in-neighbourhood} by $\InN{D}{u} \coloneqq \Set{v \in \V{D} \mid (v,u)\in\E{D}}$.
We define the \emph{out-degree} and the \emph{in-degree} of $u$ by $\Outdeg{D}{u} \coloneqq \Abs{\OutN{D}{u}}$ and $\Indeg{D}{u} \coloneqq \Abs{\InN{D}{u}}$.
We omit the index in these definitions whenever the digraph $D$ is clear from the context.
A subgraph $H$ of a digraph $D$ is a digraph with $\V{H} \subseteq \V{D}$ and $\E{H}\subseteq \E{D}$.
For an edge $e =(u,v)$ we call $u$ the \emph{tail} of $e$ and $v$ the \emph{head} of $e$.
The digraph $D$ is \emph{strongly connected} if, for every pair of vertices $u$ and $v$, there is a directed path from $u$ to $v$ in $D$ as well as a directed path from $v$ to $u$.
A maximal strongly connected subgraph of $D$ is called a \emph{strong component} of $D$.
The edge $(u,v)$ is \emph{butterfly contractible} if it is the only edge with tail $u$ or the only edge with head $v$.
A \emph{butterfly minor} of $D$ is a digraph $D'$ obtained from a subgraph of $D$ by contracting butterfly contractible edges; we write $D' \butterfly D$.

The following easy \namecref{lem:walkInMinor} about butterfly contractions shows that they cannot create new closed walks, which we make use of throughout the paper.

\begin{observation}\label{lem:walkInMinor}
	Let $D'$ be obtained from a digraph $D$ by contracting a butterfly contractible edge $(u,v)$ into the vertex $x$.
	If there is a closed walk $W'$ in $D'$ containing $x$, then there is a closed walk $W$ in $D$ containing all vertices of $V(W') - \{x\}$, and $u$ or $v$ or both.
    Specifically, if $\Outdeg{}{u} = 1$, then $W$ always contains $v$; otherwise, $W$ always contains $u$.
\end{observation}

\paragraph{Butterfly models}
As for undirected minors, there is a different perspective on butterfly minors in terms of \textsl{models}.
An \emph{arborescence} is obtained from an undirected tree by choosing a vertex and directing all edges towards it (\emph{in-branching}) or away from it (\emph{out-branching}).
If a digraph $D'$ is a butterfly minor of a digraph $D$, then there exists a function $\mu$ that assigns to every edge $e \in E(D')$ an edge $e \in E(D)$ and to every $v \in V(D')$ a subgraph $\mu(v) \subseteq D$ such that
	\begin{itemize}	
		\item $\mu(u)$ and $\mu(v)$ are vertex disjoint subgraphs of $H$ for any $u \neq v \in V(D')$,
		\item for all $e = (u,v) \in E(D')$, the edge $\mu(e)$ has its tail in $\mu(u)$ and its head in $\mu(v)$, and
		\item for all $v \in V(D')$, $\mu(v)$ is the union of an in-branching $T_i$ and an out-branching $T_o$, which only have their roots in common, such that for every $e \in E(D')$, if $v$ is the head of $e$, then the head of $\mu(e)$ is in $T_i$, and if $v$ is the tail of $e$, then the tail of $\mu(e)$ is in $T_o$.
	\end{itemize}
Such a function $\mu$ is called a \emph{butterfly model}, and its existence is equivalent to $G$ containing $H$ as a butterfly minor; see~\cite{amiri_erdos-posa_2016} for details.

Before we state the formal definition(s) for directed \treewidth in \cref{sec:dtw_definitions}, we talk about a few related concepts.

\subsection{Cops-and-robber games}

\emph{Cops-and-robber games}, also known as \emph{graph searching games}, are a form of pursuit-evasion games played on a graph or a digraph (see \cite{gradel_graph_2011,fomin_annotated_2008} for surveys).
There are many variants of these games corresponding to different graph parameters or classes; here, we concentrate on the version introduced by Johnson, Robertson, Seymour, and Thomas~\cite{johnson_directed_2001}, which corresponds to their definition of directed \treewidth. 
In the game, cops try to capture the fugitive robber with as few cops as possible.
Each cop and the robber can occupy at most one vertex of a given digraph $D$, and the robber starts the game by occupying one.
A current game \emph{position} is denoted by a pair $(C,v)$, where $C \subseteq V(D)$ is the set of vertices occupied by cops, called \emph{cop position}, and $v \in V(D)$ is the vertex occupied by the robber, called \emph{robber position}. So the start position is $(\emptyset, v)$ for some $v \in V(D)$. 

The game is played in rounds, and in each round with the current position $(C,v)$, the cops first announce their new position $C' \subseteq V(D)$. The robber can escape to any $v' \in V(D)$ in the same strong component of $D - (C \cap C')$ as $v$, i.e.~he can move to $v'$ along a directed cop-free path in $D - (C \cap C')$ only if there exists a directed cop-free path from $v'$ to $v$ as well.
Finally, the cops are placed on $C'$; this completes a round, and the new position is $(C',v')$.
A \emph{play} in $D$ is a sequence $P = (C_0,v_0),(C_1,v_1),\dots$ of game positions, where each of the robber's moves adheres to the rules described above.
If a cop position $C_i$ contains the robber position $v_i$ in the $i$-th round for some $i \in \N$, the cops catch the robber and win the play.
Otherwise, the robber can escape forever and win. Cops win trivially on any given digraph by placing cops on every vertex.
Hence, an interesting factor of the game is the minimal number of cops needed to win on a digraph.

A \emph{strategy for $k$ cops} on a digraph $D$ is a function $f_c : [V(D)]^{\leq k} \times V(D) \rightarrow [V(D)]^{\leq k}$ that assigns a cop position $C'$ to each game position $(C,v)$. 
A play $P = (C_0,v_0),(C_1,v_1),\dots$ is \emph{consistent with} $f_c$ if $C_{i+1} = f_c(C_i,v_i)$ for all $i$. If the cops win in every play $P$ consistent with $f_c$, we say $f_c$ is a \emph{winning strategy for $k$ cops}.
Given a play $P = (C_0,v_0),(C_1,v_1),\dots$, the \emph{robber space} $R_i \subseteq V(D)$ for each $i$ is a strong component of $D - C_i$ such that $R_{i-1}$ and $R_i$ are contained in the same strong component of $D - (C_{i-1} \cap C_i)$, and we let $R_0 = V(D)$.
Then, it is clear that $v_i \in R_i$ for all $i$.
A play $P = (C_0,v_0),(C_1,v_1),\dots$ is called \emph{cop-monotone} if for all $i < j < k$ we have $C_i \cap C_k \subseteq C_j$, i.e.~the cops never reoccupy vertices.
On the other hand, a play is called \emph{robber-monotone} if $R_i \supseteq R_{i+1}$ for all $i$, i.e.~if a vertex is not available to the robber once, then it remains unavailable for the rest of the play.
A strategy $f_c$ for $k$ cops is robber/cop-monotone if every play consistent with $f_c$ is robber/cop-monotone.

Johnson, Robertson, Seymour, and Thomas established the following connection between this game and directed \treewidth. 

\begin{lemma}[\cite{johnson_directed_2001}] \label{thm:c&r_game}
	Let $D$ be a digraph and $k \geq 1$.
	\begin{enumerate}
		\item If $\dtw{D} < k$, then $k$ cops have a winning strategy in the cops-and-robber game on $D$.
		\item If $k$ cops have a winning strategy in the cops-and-robber game on $D$, monotone or not, then $\dtw{D} \leq 3k+1$.
		\item If $k$ cops have a winning strategy in the cops-and-robber game on $D$, then $3k+2$ cops have a robber-monotone winning strategy on D.
	\end{enumerate}
\end{lemma}

\subsection{Obstructions}

We give a short overview of some known obstructions for directed \treewidth.

\paragraph{$k$-linked sets}
    Let $W$ be a set of vertices in a digraph $D$. 
	A \emph{balanced $W$-separator} is a set $S \subseteq V(D)$ such that every strong component of $D-S$ contains at most $\frac{|W|}{2}$ vertices of $W$. The \emph{order} of the separator is $|S|$.
	A set $W \subseteq V(D)$ is \emph{$k$-linked} if $D$ does not contain a balanced $W$-separator of order $k$.

\begin{lemma}[Reed~\cite{reed_introducing_1999}]
    \label{lem:lemma_balanced_w_separator}
	Let $D$ be a digraph. If $\dtw{D} \leq k-1$, then every set $W \in V(D)$ has a balanced $W$-separator of order at most $k$, i.e.~$D$ does not contain a $k$-linked set.
\end{lemma}

\begin{lemma}[Johnson, Robertson, Seymour, Thomas~\cite{johnson_directed_2001}] 
    \label{thm:theorem_k_linked_set}
	Every digraph $D$ either has $\dtw{D} \leq 3k + 1$ or contains a $k$-linked set, which witnesses that $\dtw{D} \geq k$.
\end{lemma}

\paragraph{Havens} \label{subsection:section_havens}
	A \emph{haven} in a digraph $D$ of order $k$ is a function $h : [V(D)]^{< k} \rightarrow \mathcal{P}(V(D))$ assigning to every set $X \subseteq V(D)$ with $|X| < k$ the vertex set of a strong component of $D-X$ such that if $Y \subseteq X \subseteq V(D)$ with $|X| < k$, then $h(X) \subseteq h(Y)$.

\begin{theorem}[\cite{johnson_directed_2001}] \label{thm:theorem_haven}
	Let $D$ be a digraph and $k \geq  1$.
	\begin{enumerate}
		\item If $k$ cops have a winning strategy in the cops-and-robber game on $D$, then $D$ has no haven of order $k+1$.
		\item If $D$ has a haven of order $k+1$, then $\dtw{D} \geq k$.
		\item If $D$ has no haven of order $k+1$, then $\dtw{D} \leq 3k+1$.
	\end{enumerate}
\end{theorem}

\paragraph{Brambles} \label{subsection:section_brambles}

	A \emph{(strong) bramble} in a digraph $D$ is a set $\mathcal{B}$ of strongly connected subgraphs of $D$ such that if $B, B' \in \mathcal{B}$, then $V(B) \cap V(B') \neq \emptyset$.
	A \emph{cover} or \emph{hitting set} of $\mathcal{B}$ is a set $X \subseteq V(D)$ such that $X \cap V(B) \neq \emptyset$ for all $B \in \mathcal{B}$. The \emph{order} of $\mathcal{B}$ is the minimum size of a cover of $\mathcal{B}$. The \emph{bramble number} of $D$, denoted $\bn(D)$, is the maximum order of any bramble in $D$.

Note that this concept is often called a \emph{strong bramble} in the literature in contrast to a \emph{(weak) bramble}~\cite{reed_introducing_1999} that does not demand the elements to overlap but also allows them to be connected by edges in both directions.
This weaker notion, however, is not closed under taking butterfly minors, as can be seen, in~\cref{fig:bramble}.
Thus, we work with strong brambles.

\begin{figure}[!htb]
	\centering
	\begin{subfigure}[t]{.45\linewidth}
	\centering
	\begin{tikzpicture}
		[
		state/.style={circle, draw=red, thick, minimum size=4mm, font=\small,inner sep=0pt},
		stateb/.style={circle, draw=blue, thick, minimum size=4mm, font=\small,inner sep=0pt},
		blank/.style={circle, ultra thin, minimum size=10mm},
		bubble/.style={circle, draw=cyan, thick, fill=cyan!40, opacity=0.2, minimum size=5mm, font=\small,inner sep=26pt},
		> = latex, 
		shorten > = 1pt, 
		auto,
		node distance = 1.5cm, 
		line width= 0.3mm
		,scale=.9,transform shape
		]
		
		\node[state] (1) {$1$};
		\node[state] (2) [below left=0.7cm and 1cm of 1] {$2$};
		\node[state] (3) [below=1.5cm of 1] {$3$};
		\node[stateb] (4) [right=2cm of 1]{$4$};
		\node[stateb] (5) [below right=0.7cm and 1cm of 4] {$5$};
		\node[stateb] (6) [below=1.5cm of 4] {$6$};
		
		\node[blank,red] (c1) [below left=0.5cm and 0cm of 1] {$C_1$};
		\node[blank,blue] (c2) [below right=0.5cm and 0cm of 4] {$C_2$};
	
		\path[->,red] (1) edge node {} (2);
		\path[->,red] (2) edge node {} (3);
		\path[->,red] (3) edge node {} (1);
		\path[->,blue] (4) edge node {} (5);
		\path[->,blue] (5) edge node {} (6);
		\path[->,blue] (6) edge node {} (4);
		\path[->] (1) edge node {} (4);
		\path[->] (6) edge node {} (3);
		
	
	\end{tikzpicture}
	\subcaption{A digraph $D$.}
	\label{fig:bramble1}
\end{subfigure}%
\hfill
\begin{subfigure}[t]{.45\linewidth}
	\centering
	\begin{tikzpicture}
		[
		state/.style={circle, draw=black, thick, minimum size=4mm, font=\small,inner sep=0pt},
		stater/.style={circle, draw=red, thick, minimum size=4mm, font=\small,inner sep=0pt},
		stateb/.style={circle, draw=blue, thick, minimum size=4mm, font=\small,inner sep=0pt},
		state2/.style={circle, draw=black, thick, minimum size=3.5mm, font=\small,inner sep=0pt},
		state2r/.style={circle, draw=red, thick, minimum size=3.5mm, font=\small,inner sep=0pt},
		state2b/.style={circle, draw=blue, thick, minimum size=3.5mm, font=\small,inner sep=0pt},
		blank/.style={circle, ultra thin, minimum size=10mm},
		bubble/.style={circle, draw=blue, thick, fill=cyan!40, opacity=0.2, minimum size=5mm, font=\small,inner sep=26pt},
		> = latex, 
		shorten > = 1pt, 
		auto,
		node distance = 1.5cm, 
		line width= 0.3mm
		,scale=.9,transform shape
		]
		
		\node[stater] (1) {$1$};
		\node[state2r] (1') [below left=0.18cm and 0.33cm of 1] {$a$};
		\node[stater] (2) [below left=0.7cm and 1cm of 1] {$2$};
		\node[state2r] (2') [above left=0.18cm and 0.33cm of 3] {$b$};
		\node[stater] (3) [below=1.5cm of 1] {$3$};
		
		\node[state2] (a) [right=0.8cm of 1]{$c$};
		\node[state2] (b) [right=0.8cm of 3]{$f$};
		
		\node[stateb] (4) [right=2cm of 1]{$4$};
		\node[state2b] (4') [below right=0.18cm and 0.33cm of 4] {$d$};
		\node[stateb] (5) [below right=0.7cm and 1cm of 4] {$5$};
		\node[state2b] (5') [above right=0.18cm and 0.33cm of 6] {$e$};
		\node[stateb] (6) [below=1.5cm of 4] {$6$};
		
		\path[->,red] (1) edge node {} (1');
		\path[->,red] (1') edge node {} (2);
		\path[->,red] (2) edge node {} (2');
		\path[->,red] (2') edge node {} (3);
		\path[->,red] (3) edge node {} (1);
		
		\path[->,blue] (4) edge node {} (4');
		\path[->,blue] (4') edge node {} (5);
		\path[->,blue] (5) edge node {} (5');
		\path[->,blue] (5') edge node {} (6);
		\path[->,blue] (6) edge node {} (4);
		
		\path[->] (1) edge node {} (a);
		\path[->] (a) edge node {} (4);
		\path[->] (6) edge node {} (b);
		\path[->] (b) edge node {} (3);
		
		\node[blank,red] (c1) [below left=0.5cm and 0cm of 1] {$C_1'$};
		\node[blank,blue] (c2) [below right=0.5cm and 0cm of 4] {$C_2'$};

	\end{tikzpicture}
	\subcaption{A digraph $D'$.}
	\label{fig:bramble2}
\end{subfigure}%
	\caption{The digraph $D$ on the left is a butterfly minor of the digraph $D'$ to the right.
        However, $D$ contains a weak bramble of order $2$ while $D'$ has no weak bramble of order $2$.}
	\label{fig:bramble}
\end{figure}
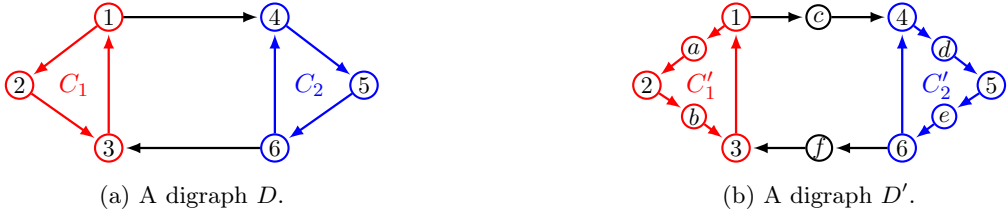

We know about the following relations of brambles to directed \treewidth.
It is important to note that the proofs of both \cref{lem:lemma_bramble_k_linked_set,lem:haven_to_bramble} indeed yield \textsl{strong} brambles.

\begin{lemma} [\cite{reed_introducing_1999}] \label{lem:lemma_bramble_k_linked_set} 
	Let $D$ be a digraph. If $D$ contains a $k$-linked set, it contains a bramble of order $k+1$.
\end{lemma}

\begin{lemma} [\cite{safari_d-width_2005}] \label{lem:lemma_bramble_haven}
	Let $D$ be a digraph. If $D$ contains a bramble of order $k$, it contains a haven of order $k$.
\end{lemma}

\begin{lemma} [\cite{safari_d-width_2005}]
    \label{lem:haven_to_bramble}
	Let $D$ be a digraph. If $D$ contains a haven of order $k+1$, it contains a bramble of order greater than $\frac{k}{2}$.
\end{lemma}

\begin{corollary} \label{cor:bramble_dtw}
	Let $D$ be a digraph and $k \geq  1$.
	\begin{enumerate}
		\item If $D$ contains a bramble of order $k+1$, then $\dtw{D} \geq k$.
		\item If $\dtw{D} > 3k+1$, then $D$ contains a bramble of order $k+1$.
	\end{enumerate}
\end{corollary}

\section{Brambles are closed under butterfly minors}
\label{sec:brambles_closed}

As seen, containing a bramble of high order is equivalent to having high directed treewidth up to small functions.
In this section, we show that the property of containing a bramble of high order is closed under taking butterfly minors (\cref{thm:bn_thm}).
To do so, we introduce the concept of a \emph{major graph}, which can be considered the opposite of a minor.
	Let $D$, and $D'$ be digraphs.
	If $D' \preccurlyeq_b D$, then we call $D$ a \emph{major graph} of $D'$.
	If $D' \preccurlyeq_b D$ and $D' \not\preccurlyeq_b X$ holds for all $X \subsetneq D$, then we call $D$ a \emph{minimal major graph} of $D'$.

\begin{observation}\label{lem:minimalMajorModel}
	Let $D$, $D'$ be digraphs such that $D' \preccurlyeq_b D$. Then there is a minimal major graph $H$ of $D'$ such that $H \subseteq D$. Furthermore, if $S'$ is a subgraph of $D'$, then there is a minimal major graph $S$ of $S'$ such that $S \subseteq H$.
\end{observation}

We next show that no two non-trivial strong components can be combined into one by butterfly contraction.

\begin{lemma} \label{lem:butterfly_contracting_edge_between_sc}
	Let $D$ be a digraph with at least two strong components and $D'$ be a digraph obtained from $D$ by contracting a butterfly contractible edge with endpoints in different strong components of $D$.
	Then, for every strong component $C'$ of $D'$, there is a strong component $C$ of $D$ such that $C'$ is isomorphic to $C$.
\end{lemma}

\begin{proof}
    Let $(u,v)$ be the edge contracted to obtain $D'$ from $D$ and let $C_u$ and $C_v$ be the distinct strong components of $D$ such that $u \in V(C_u)$ and $v \in V(C_v)$.
	  Assume that $\Outdeg{D}{u} = 1$. (The other case, i.e.~$\Indeg{D}{v} = 1$, is analogous.)
    
	If $|V(C_u)| \geq 2$, then we have $\Outdeg{C_u}{u} \geq 1$, and hence $\Outdeg{D}{u} \geq 2$, a contradiction.
	Therefore, we have $|V(C_u)| = 1$.
    There is nothing to show if there is no edge $e'$ such that $u$ is incident with $e'$ and $e' \neq (u,v)$. Therefore, we assume at least one such edge $e'$ exists. 
	Then $u$ is the head of every such edge $e'$, and the tail of $e'$ lies in a strong component $\hat{C}$ of $D$ with $C_v \neq \hat{C} \neq C_u$.
	
	Suppose there is a strong component $C'$ of $D'$ such that there is no strong component $C$ of $D$ such that $C'$ is isomorphic to $C$.
    Then, $C'$ is obtained by contracting $(u,v)$ into a vertex $x$.
    Since $C_v$ and $D'[(V(C_v) - \{v\}) \cup \{x\}]$ are isomorphic and strongly connected, we have $(V(C_v) - \{v\}) \cup \{x\} \subsetneq V(C')$, and there is at least one vertex $w \in V(C') - ((V(C_v) - \{v\}) \cup \{x\})$.
	Then there is an edge $(w,x) \in V(C')$, and we have $w \in \hat{C}$, where $\hat{C}$ is a strong component of $D$ with $C_v \neq \hat{C} \neq C_u$.
	Since $C'$ is a strong component of $D'$, there is a closed walk $W'$ in $D'$ containing $x$ and $w$.
    Then, by~\cref{lem:walkInMinor}, there is a closed walk in $D$ that contains $v$ and $w$, which is a contradiction. 
\end{proof}

The following lemma was implicitly used in \cite{amiri_erdos-posa_2016}.

\begin{lemma}[\cite{amiri_erdos-posa_2016}] \label{lem:sc_minimalMajorModel}
	Let $D$, $D'$ be digraphs such that $D' \preccurlyeq_b D$.
	If $D'$ is strongly connected, and $D$ is a minimal major graph of $D'$, then $D$ is also strongly connected.
\end{lemma}

This allows us to establish that the bramble number is closed under taking butterfly minors.

\begin{theorem}\label{thm:bn_thm}
	Let $D$, $D'$ be digraphs such that $D' \preccurlyeq_b D$.
    Then $\bn(D') \leq \bn(D)$.
\end{theorem}

\begin{proof}
	By \cref{lem:minimalMajorModel}, there is a minimal major graph $H$ of $D'$ such that $H \subseteq D$.
    Let $\mu$ be a butterfly model of $D'$ in $H$.
	Let $\mathcal{B'} = \{B_1', \dots, B_n'\}$ be a bramble of $D'$ of maximum order and $C' = \{c_1', \dots, c_m'\}$ be a cover of $\mathcal{B'}$ of minimum size. Then we have $|C'| = \bn(D')$. For each $B_i' \in \mathcal{B'}$, we choose $B_i$ and $\mu_i$ such that $B_i$ is a minimal major graph of $B_i'$ with $B_i \subseteq H$, $\mu_i$ is a tree-like model of $B_i'$ in $B_i$, and $B_i$ and $\mu_i$ are obtained from $\mu(D')$ by following the proof of \cref{lem:minimalMajorModel}. Then by construction, we have $\mu_i(v') \subseteq \mu(v')$ for every $v' \in V(B_i')$.
	We define $\mathcal{B} \coloneqq \{B_1, \dots, B_n\}$. 
	
	We claim that $\mathcal{B}$ is a bramble of $D$.
    By \cref{lem:minimalMajorModel}, there exists a corresponding minimal major graph $B_i$ for each $B_i'$ with $B_i \subseteq H \subseteq D$.
    By \cref{lem:sc_minimalMajorModel}, $B_1, \dots, B_n$ are strongly connected subgraphs of $D$.
	Assume that there exists $B_i' \in \mathcal{B'}$ with $|B_i'| = 1$. Since $\mathcal{B'}$ is a bramble, $C' = \{v\}$ for $v \in B_i'$ is a cover of $\mathcal{B'}$. Hence, we have $\bn(D') = |C'| = 1$, and we can always find a bramble of order $1$ in $D$. Therefore, we assume $|B_i'| \geq 2$ for all $i \in \{1,\dots,n\}$. 
	
	We want to show that $V(B_k) \cap V(B_l) \neq \emptyset$ holds for all $B_k, B_l \in \mathcal{B}$. Let $B_k, B_l \in \mathcal{B}$. Since $V(B_k') \cap V(B_l') \neq \emptyset$ for all $B_k', B_l' \in \mathcal{B'}$, there is at least one vertex $v' \in V(B_k') \cap V(B_l')$.
	Since $|B_i'| \geq 2$ for all $i$, and $B_i'$ is strongly connected, every $x \in V(B_i')$ has at least one outgoing and one ingoing edge in $B_i'$, i.e.~both the in-branching and the out-branching of $\mu_i(x)$ contain at least one leaf.
	Therefore, by the choice of $B_i$ and $\mu_i$, for every $B_i'$ with $v' \in V(B_i')$, $\mu_i(v')$ contains the root of $\mu(v')$ where the in-branching and the out-branching meet. This implies that there is always a vertex $v \in V(\mu_k(v')) \cap V(\mu_l(v')) \subseteq V(B_k) \cap V(B_l)$.
	
	Due to the above argument, we can find a vertex $c_j \in V(D)$ for each $c_j' \in C'$ such that if $c_j'$ covers $B_{i_1}', \dots, B_{i_l}'$ for $1 \leq l \leq n$, then $c_j \in \bigcap_{1 \leq k \leq l}V(\mu_{i_k}(c_j')) \subseteq \bigcap_{1 \leq k \leq l}V(B_{i_k})$. We define $C \coloneqq \{c_1, \dots, c_m\}$, where for each $c_j' \in C'$, $c_j \in C$ is chosen as above. We claim that $C$ is a cover of $\mathcal{B}$ of minimum size.
	By construction, $C$ is a cover of $\mathcal{B}$. Towards a contradiction, suppose there is a cover $C''$ of $\mathcal{B}$ such that $|C''| < |C|$. Since $B_i$ is a minimal major graph of $B_i'$, every $v \in V(B_i)$ is contained in $\mu_i(v')$ for some $v' \in V(B_i')$. If $c \in C''$ covers $B_{i_1}, \dots, B_{i_l}$ for $1 \leq l < m$, i.e.~$c \in \bigcap_{1 \leq k \leq l}V(B_{i_k})$, then $c \in \bigcap_{1 \leq k \leq l} V(\mu_{i_k}(v'))$ for some $v' \in \bigcap_{1 \leq k \leq l}V(B_{i_k}')$ (by the choice of $B_i$ and $\mu_i$ for all $i$), i.e.~there is $v'$ that covers every $B_{i_1}', \dots, B_{i_l}'$. Therefore, we can find a cover of $\mathcal{B'}$ of size less than $|C'| = \bn(D')$, which is a contradiction.
\end{proof} 

\section{Definitions of directed tree decompositions}
\label{sec:dtw_definitions}

We provide a base definition containing the properties that all the different definitions capturing directed \treewidth we consider have in common.

\begin{definition} [Abstract Digraph Decomposition]	\label{def:abstract}
	An \emph{abstract digraph decomposition} of a digraph $D$ is a triple $\mathcal{T} \coloneqq $ $(T, \beta, \gamma)$, where $T$ is a rooted directed tree, $\beta : V(T) \rightarrow 2^{V(D)}$ and $\gamma : E(T) \rightarrow 2^{V(D)}$ such that $\bigcup \{\beta(t) : t \in V(T)\} = V(D)$.\\
	For every $t \in V(T)$, we define $\Gamma(t) \coloneqq \beta(t) \cup \bigcup_{e \sim t}\gamma(e)$ and $T_t \coloneqq T[\{ t' \in V(T) : t'$ is reachable from $t$ by a directed path in $T\}]$. Furthermore, for a subtree $S \subseteq T$, we define $\beta(S) \coloneqq \bigcup_{t \in V(S)}\beta(t)$.

    \vspace{3pt}
	The \emph{width} $w(\mathcal{T})$ of $\mathcal{T}$ is max$\{|\Gamma(t)| - 1 : t \in V(T)\}$.
	For $t \in V(T)$ and $e \in E(T)$, we call $\beta(t)$ a \emph{bag} and $\gamma(e)$ a \emph{guard}.
\end{definition}

Based on \cref{def:abstract}, we define five different versions of directed tree decompositions and directed tree-width that can be found in the literature.

\begin{definition}[\dtdNW \cite{johnson_directed_2001}]
	\label{def:1dtd}
	An \emph{\dtdNW} of a digraph $D$ is  an abstract digraph decomposition $\mathcal{T} \coloneqq (T, \beta, \gamma)$ such that
	\begin{enumerate}[leftmargin=1.5cm,labelindent=16pt,label= ($\dtwNW{}$\arabic*)]
		\item $\{\beta(t) : t \in V(T)\}$ is a partition of $V(D)$ into nonempty sets, and
		\item \label{D1.2}  for all $e = (s,t) \in E(T)$, $\beta(T_t) \subseteq V(D) - \gamma(e)$ and there is no walk in $D-\gamma(e)$ with first and last vertices in $\beta(T_t)$ that uses a vertex of $V(D)-(\beta(T_t) \cup \gamma(e))$.
	\end{enumerate}
	The \emph{\tdtwNW} of $D$, denoted by \dtwNW{D}, is min$\{w(\mathcal{T}) : \mathcal{T}$ is an \dtdNW\ of $D\}$.
\end{definition}
\dtwNW{} stands for `No Walk'. If \ref{D1.2} holds, the vertex set $\beta(T_t)$ is called \emph{$\gamma(e)$-normal}.

\begin{definition} [\dtdNCW \cite{giannopoulou_directed_2022}]
	\label{def:2dtd}
	An \emph{\dtdNCW} of a digraph $D$ is  an abstract digraph decomposition $\mathcal{T} \coloneqq (T, \beta, \gamma)$ such that
	\begin{enumerate}[leftmargin=1.7cm,labelindent=16pt,label= ($\dtwNCW{}$\arabic*)]
		\item $\{\beta(t) : t \in V(T)\}$ is a partition of $V(D)$ into nonempty sets, and
		\item \label{D2.2} for all $e = (s,t) \in E(T)$, there is no closed walk in $D-\gamma(e)$ containing a vertex of $\beta(T_t)$ and a vertex of $V(D)- \beta(T_t)$.
	\end{enumerate}
	The \emph{\tdtwNCW} of $D$, denoted by \dtwNCW{D}, is min$\{w(\mathcal{T}) : \mathcal{T}$ is an \dtdNCW\ of $D\}$.
\end{definition}
\dtwNCW{} stands for `No Closed Walk'. The above definition is slightly different from \dtdNWs. For some $e = (s,t) \in E(T)$, $\beta(T_t)$ may contain some vertices of $\gamma(e)$. Moreover, there may be an unclosed walk in $D - \gamma(e)$ with the first and last vertices in $\beta(T_t)$ that uses a vertex of $V(D)-(\beta(T_t) \cup \gamma(e))$. By simply allowing empty bags, we obtain the following definition.

\begin{definition} [\dtdNCWE \cite{giannopoulou_directed_2022}]
	\label{NCWE}
	An \emph{\dtdNCWE} of a digraph $D$ is  an abstract digraph decomposition $\mathcal{T} \coloneqq (T, \beta, \gamma)$ such that
	\begin{enumerate}[leftmargin=1.9cm,labelindent=16pt,label= ($\dtwNCWE{}$\arabic*)]
		\item\label{NCWE_1} $\{\beta(t) : t \in V(T)\}$ is a partition of $V(D)$ into possibly empty sets such that $\beta(r) \neq \emptyset$, where $r$ is the root of $T$, and
		\item\label{NCWE_2} for all $e = (s,t) \in E(T)$, there is no closed walk in $D-\gamma(e)$ containing a vertex of $\beta(T_t)$ and a vertex of $V(D)- \beta(T_t)$.
	\end{enumerate}
	The \emph{\tdtwNCWE} of $D$, denoted by \dtwNCWE{D}, is min$\{w(\mathcal{T}) : \mathcal{T}$ is an \dtdNCWE of $D\}$.
\end{definition}
\dtwNCWE{} stands for `No Closed Walk' with possibly empty bags. By definition, any \dtdNCW is an \dtdNCWE. 

\begin{definition} [\dtdSCE \cite{johnson_addendum_2001}]
	\label{def:3dtd}
	An \emph{\dtdSCE} of a digraph $D$ is  an abstract digraph decomposition $\mathcal{T} \coloneqq (T, \beta, \gamma)$ such that
	\begin{enumerate}[leftmargin=1.5cm,labelindent=16pt,label= ($\dtwSCE{}$\arabic*)]
		\item \label{D3.1} $\{\beta(t) : t \in V(T)\}$ is a partition of $V(D)$ into possibly empty sets such that $\beta(r) \neq \emptyset$, where $r$ is the root of $T$,
		\item \label{D3.2} for all $e = (s,t) \in E(T)$, $\beta(T_t)$ is the vertex set of a strong component of $D-\gamma(e)$, and
		\item \label{D3.3} $|V(T)| \leq |V(D)|^2$.
	\end{enumerate}
	The \emph{\tdtwSCE} of $D$, denoted by \dtwSCE{D}, is min$\{w(\mathcal{T}) : \mathcal{T}$ is an \dtdSCE\ of $D\}$.
\end{definition}
\dtwSCE{} stands for `Strong Component' with possibly empty bags.

\begin{definition} [\dtdSCD \cite{kreutzer_width-measures_2014}]
	\label{def:4dtd}
	An \emph{\dtdSCD} of a digraph $D$ is  an abstract digraph decomposition $\mathcal{T} \coloneqq (T, \beta, \gamma)$ such that
	\begin{enumerate}[leftmargin=1.6cm,labelindent=16pt,label= ($\dtwSCD{}$\arabic*)]
		\item \label{D1.4} $\{\beta(t) : t \in V(T)\}$ is a partition of $V(D)$ into nonempty sets,
		\item \label{D2.4} for all $e = (s,t) \in E(T)$, $\beta(T_t)$ is the vertex set of a strong component of $D-\gamma(e)$, and
		\item \label{D3.4} if $t \in V(T)$ and $t_1,...,t_l$ are the children of $t$ in $T$, then $\bigcup_{1 \leq i \leq l}\beta(T_{t_i}) \cap \bigcup_{e \sim t}\gamma(e) = \emptyset$.
	\end{enumerate}
	The \emph{\tdtwSCD} of $D$, denoted by \dtwSCD{D}, is min$\{w(\mathcal{T}) : \mathcal{T}$ is an \dtdSCD\ of $D\}$.
\end{definition}
\dtwSCD{} stands for `Strong Component' with $\bigcup_{1 \leq i \leq l}\beta(T_{t_i})$ and $\bigcup_{e \sim t}\gamma(e)$ being disjoint for each $t \in V(T)$ and its children $t_1, \dots, t_l$. Recall that if $\dtw{D} < k$, then $k$ cops have a winning strategy in the cops-and-robber game on $D$ (\cref{thm:c&r_game}). The following observation shows the differences between \dtdSCD and the other definitions.

\begin{observation}
     If $\dtwSCD{D} < k$, then $k$ cops have a robber-monotone winning strategy; otherwise, the winning strategy provided by \cref{thm:c&r_game} is not necessarily robber-monotone.
\end{observation}

As a reminder: If a statement holds for the directed \treewidth with respect to every definition mentioned above, the directed \treewidth of a digraph $D$ is denoted by $\dtw{D}$. 

\begin{table} [!ht]
	\centering
    \begin{tabular}{ M{7cm}||M{1.1cm} M{1.1cm} M{1.1cm} M{1.1cm} M{1.1cm} }
        \toprule[0.5mm]
		 & \multicolumn{5}{c}{\small directed tree decompositions} \\
		 & \dtwNW{} & \dtwNCW{} & \dtwNCWE{} & \dtwSCE{} & \dtwSCD{}\\
		\hline \hline
		$\beta(T_t)$ contains at most one strong component of $D - \gamma(e)$ & & & & $\checkmark$ & $\checkmark$ \\ \hline
		bags are non-empty sets & $\checkmark$ & $\checkmark$ & & & $\checkmark$ \\ \hline
		$\beta(T_t)$ is disjoint from $\gamma(e)$ & $\checkmark$ & & & $\checkmark$ & $\checkmark$\\
		\bottomrule[0.5mm]
	\end{tabular}
	\caption{Differences between the directed tree decompositions. We let $\mathcal{T} \coloneqq (T, \beta, \gamma)$ be a directed tree decomposition of a digraph $D$ corresponding to each column and $e =(s,t) \in E(T)$. The check mark signifies that the property always holds, while the blank space indicates that the property does not necessarily hold.}\label{tab1}
\end{table}

\begin{observation} \label{observation_SC}
	Let $D$ be a digraph, $\mathcal{T} \coloneqq (T, \beta, \gamma)$ be a directed tree decomposition of $D$ and  $e = (s,t) \in E(T)$.
	\begin{enumerate}
		\item If $\mathcal{T}$ is an \dtdNW, $\beta(T_t)$ is the union of vertex sets of some strong components of $D - \gamma(e)$.
		\item If $\mathcal{T}$ is an \dtwNCW{}- or an \dtdNCWE, $\beta(T_t) - \gamma(e)$ is the union of vertex sets of some strong components of $D - \gamma(e)$.
		\item If $\mathcal{T}$ is an \dtwSCE{}- or an \dtdSCD, $\beta(T_t)$ is the vertex set of a strong component of $D - \gamma(e)$.
	\end{enumerate}
\end{observation}

The following lemma follows directly from the above observation.

\begin{lemma} [\cite{kreutzer_width-measures_2014}]
	\label{lem:lemma_strong_separator}
	Let $D$ be a digraph and $\mathcal{T} \coloneqq (T, \beta, \gamma)$ be a directed tree decomposition of $D$ (for any of the \cref{def:1dtd,def:2dtd,def:3dtd,def:4dtd}).
	\begin{enumerate}
		\item For every $e = (s,t) \in E(T)$, $\gamma(e)$ is a separator in $D$, i.e.~if $S_s, S_t$ are the two components of $T - e$, then every strong component of $D - \gamma(e)$ is either contained in $\beta(S_s)$ or $\beta(S_t)$.
		\item If $t \in V(T)$ and $S_1,..., S_l$ are the components of $T - t$, then every strong component of $D - \Gamma(t)$ is contained in exactly one $\beta(S_i)$ for $1 \leq i \leq l$.
	\end{enumerate}
\end{lemma}

\subsection{Overview of results}
\label{sec:result_overview}

Our main objective was to identify which of the given definitions are closed under taking butterfly minors and which are not.
In~\cref{sec:B_minor_closed_dtd}, we present the main result.
\begin{restatable}{theorem}{closedNCWE}
\label{thm:NCWE_closure}
	Let $D$, $D'$ be digraphs such that $D' \preccurlyeq_b D$. Then $\dtwNCWE{D'} \leq \dtwNCWE{D}$.
\end{restatable}
None of the other definitions is closed under taking butterfly minors.
For \tdtwNW this was established by Adler~\cite{adler_directed_2007}, and for \tdtwNCW we provide a proof in~\cref{sec:NCW_not_closed}
\begin{restatable}{theorem}{NOTClosedNCW}
    \label{thm:NCW_not_closed}
    \tdtwNCW is not closed under taking butterfly minors.
\end{restatable}
For the remaining two, we prove this in~\cref{sec:SC0andSCDnotClosed}.
\begin{restatable}{theorem}{ThmSCENotClosed}
    \label{thm:theorem_dtw3_not_minor_closed}
    \tdtwSCE is not closed under taking butterfly minors.
\end{restatable}
\begin{restatable}{theorem}{ThmSCDNotClosed}
    \label{thm:SCDNotClosed}
    \tdtwSCD is not closed under taking butterfly minors.
\end{restatable}

\begin{figure}[!ht]
    \begin{center}
    	\begin{tikzpicture}
	[
	state/.style={circle, draw=black, thick, minimum size=5mm, font=\small,inner sep=0pt},
	blank/.style={circle, ultra thin, minimum size=1mm},
	> = stealth, 
	shorten > = 1pt, 
	auto,
	node distance = 1.5cm, 
	line width= 0.2mm
	,scale=0.9, transform shape
	]

	\node[blank] (ncw) {$\mathsf{NCW}$};
	\node[blank] (sce) [below=2.5cm and 1cm of ncw] {$\mathsf{SC_{\emptyset}}$};
	\node[blank] (nw) [below right=0.7cm and 2cm of ncw] {$\mathsf{NW}$};
	\node[blank] (scd) [right=0cm and 1.3cm of nw] {$\mathsf{SC_d}$};
	\node[blank] (ncwe) [below left=0.7cm and 2cm of ncw] {$\mathsf{NCW_{\emptyset}}$};
    \node[blank] (1) [above right=-0.2cm and 2.5cm of scd] {$3\mathsf{NW} + 2$};
    \node[blank] (2) [below=-1.45cm of 1] {$3\mathsf{NCW} + 2$};
    \node[blank] (3) [below=-1.6cm of 2] {$3\mathsf{NCW_{\emptyset}} + 2$};
    \node[blank] (4) [below=-1.5cm of 3] {$3\mathsf{SC_{\emptyset}} + 2$};
    \node[blank] (5) [right=0.6cm and 1.8cm of scd] {};
	
	\path[<->] (ncw) edge node [sloped, pos=0.49, below=-0.4, align=center, fill=white!] {$\lesseqgtr$} 
                          node [pos=0.5, above=0.25, align=center, fill=white!] {\cref{thm:SCE<NCW}}
                          node [pos=0.5, below=0.2, align=center, fill=white!] {\cref{thm:NCW<SCE}} (sce);
	\path[<-] (ncw) edge node [pos=0.5, sloped, below=-0.3, fill=white!] {$\leq$}
                          node [pos=0.6, sloped, above=0.2, fill=white!] {def.} (nw);
	\path[<-] (sce) edge node [pos=0.5, sloped, below=-0.3, fill=white!] {$\leq$}
                          node [pos=0.6, sloped, below=0.2, fill=white!] {\cite{johnson_addendum_2001}} (nw);
	\path[<-] (nw) edge node [pos=0.5, sloped, below=-0.3, fill=white!] {$\leq$}
                        node [pos=0.6, sloped, above=0.2, fill=white!] {def.} (scd);
	\path[<-] (ncwe) edge node [pos=0.5, sloped, below=-0.3, fill=white!] {$\leq$} 
                          node [pos=0.6, sloped, above=0.2, fill=white!] {def.}(ncw);
	\path[<-] (ncwe) edge node [pos=0.5, sloped, below=-0.3, fill=white!] {$\leq$} 
                          node [pos=0.6, sloped, below=0.2, fill=white!] {def.}(sce);

    \path[<-] (scd) edge node [pos=0.5, sloped, below=-0.3, fill=white!] {$\leq$}
                         node [pos=0.5, above=0.2, align=center, fill=white!] {\cref{lem:comparing_dtwSCD_with_others}}(5);

\end{tikzpicture}
\caption{The relation between directed tree-width with respect to different types of directed tree decompositions. An arrow with `$\leq$' means bounded in one direction, and a bidirected arrow with `$\lesseqgtr$' means not bounded in any direction.}
\label{fig:figure_dtw_compare}
    \end{center}
\end{figure}
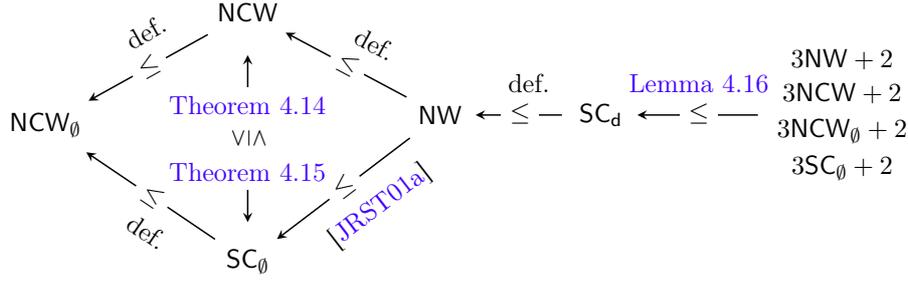

On the front of comparing the given definitions with each other, we complete the picture, which is illustrated in~\cref{fig:figure_dtw_compare}.
Most of the relations follow the definitions directly.
The proof in \cite{johnson_addendum_2001} shows that $\dtwSCE{D} \leq \dtwNW{D}$.
We present counterexamples to $\dtwNCW{D} \leq \dtwSCE{D}$ in \cref{sec:SC0_no_upper_bound_on_NCW} and $\dtwSCE{D} \leq \dtwNCW{D}$ in \cref{sec:counterexample_NCW_being_upper_bound_of_SCE}.
\begin{restatable}{theorem}{CompSCENCW}
    \label{thm:SCE<NCW}
	There is exists a digraph $D$ with $\dtwSCE{D} < \dtwNCW{D}$.
\end{restatable}
This needs some extra machinery in the form of \emph{strategy trees}, which are introduced in \cref{sec:strategy_trees}.

At the end of \cref{sec:SC0_no_upper_bound_on_NCW}, we additionally discuss that the given example graph provides evidence that the converse of the first part of \cref{cor:bramble_dtw} does not hold for \dtwNW{}-, \dtwNCW{}- and \tdtwSCD.

\begin{restatable}{theorem}{CompNCWsmallerSCE}
    \label{thm:NCW<SCE}
    There exists a digraph $D$ with $\dtwNCW{D} < \dtwSCE{D}$.
\end{restatable}

In \cref{sec:counterexample_NCW_being_upper_bound_of_SCE}, we discuss how this implies that for \tdtwSCE the converse of the second part of \cref{thm:theorem_haven}, the converse of the first part of \cref{cor:bramble_dtw}, as well as the converse of \cref{lem:lemma_balanced_w_separator} do not hold.

Additionally, we learn that the gap in \cref{thm:c&r_game} cannot be closed all the way as the graphs we analyse in \cref{sec:counterexample_NCW_being_upper_bound_of_SCE,sec:counterexample_SCE_being_upper_bound_of_NCW} show that the \dtwSCE{}-, \dtwNW{}-, \dtwNCW{}- and \tdtwSCD of a digraph $D$ is not exactly one less than the minimal number of cops needed to win in the cops-and-robber game on $D$.

Nevertheless, $\dtwNCW{D}$ and $\dtwSCE{D}$ are within a constant factor of each other by the following lemma.
\begin{lemma} \label{lem:comparing_dtwSCD_with_others}
	Let $D$ be a digraph. Then for all $\variablestyle{w} \in \{\dtwNW{}, \dtwNCW{}, \dtwNCWE{}, \dtwSCE{}\}$ holds
	\begin{align*}
		\dtwSCD{D} \leq 3\cdot\variablestyle{w}(D) + 2.
	\end{align*}
\end{lemma}
\begin{proof}
	 By \cref{lem:lemma_balanced_w_separator}, if a digraph $D$ has $\variablestyle{w}(D) \leq k$ for $\variablestyle{w} \in \{\dtwNW{}, \dtwNCW{}, \dtwNCWE{}, \dtwSCE{}\}$, then $D$ does not contain a $k+1$-linked set. Then by \cref{thm:theorem_k_linked_set}, $D$ has an \dtdSCD\ of width at most $3k+2$.
\end{proof}

Finally, in~\cref{sec:c&r_of_B_minor}, we also discuss the behaviour of winning strategies in the cops and robber game under taking butterfly minors.
\begin{restatable}{theorem}{ThmCopsAndRobbermonotonenotclosed}
\label{thm:c&r_game_closure}
    The number of cops needed to win the robber-monotone cops and robber game is not closed under taking butterfly minors.
\end{restatable}

\section{Directed \treewidth that is closed under taking butterfly minors}\label{sec:B_minor_closed_dtd}

Here, we prove that \dtdNCWEs are closed under taking butterfly minors (\cref{thm:NCWE_closure}).
The following \namecref{lem:lemmaDifferentRoot} shows that a given \dtdNCWE\ is robust with respect to choosing a different root vertex.

\begin{lemma}\label{lem:lemmaDifferentRoot}
	Let $D$ be a digraph and $\mathcal{T} \coloneqq (T, \beta, \gamma)$ be an \dtdNCWE\ of $D$ of width $k$. Let $r \in V(T)$ be the root of $T$ and $r' \in V(T)$ such that $r' \neq r$ and $\beta(r') \neq \emptyset$. Then there is an \dtdNCWE\ $\mathcal{T'} \coloneqq (T', \beta', \gamma')$ of $D$ of width $k$ such that $r'$ is the root of it.
\end{lemma}

\begin{proof}
	We construct $\mathcal{T'} \coloneqq (T', \beta', \gamma')$ as follows: for all $v \in V(T)$, we put $v \in V(T')$ and $\beta'(v) \coloneqq \beta(v)$. If $e = (s,t) \in E(T)$ is on the $(r,r')$-path in $T$, put $e' = (t,s) \in E(T')$ and $\gamma'(e') \coloneqq \gamma(e)$, otherwise put $e \in E(T')$ and $\gamma'(e) \coloneqq \gamma(e)$, i.e. every edge on the $(r,r')$-path is oriented away from $r'$ in $T'$. Then, $T'$ has $r'$ as its root.
	
	We claim that $\mathcal{T'}$ is an \dtdNCWE\ of $D$ of width $k$. As $\mathcal{T}$ satisfies \ref{NCWE_1}, $\mathcal{T'}$ satisfies the condition by construction. Since for each $v \in V(T')$, $\Gamma(v)$ is the same as in $T$, the width of $\mathcal{T'}$ is $k$. For every $e=(s,t) \in E(T')$ that is not on the $(r,r')$-path in $T$, it holds that $\gamma'(e) = \gamma(e)$, $\beta'(t) = \beta(t)$ and consequently, $\beta'(T'_t) = \beta(T_t)$. Therefore, there is no closed walk in $D - \gamma'(e)$ containing a vertex of $\beta'(T'_t)$ and a vertex of $V(D)- \beta'(T'_t)$.
	
	Let $e=(s,t) \in E(T')$ such that $e'=(t,s) \in E(T)$ is on the $(r,r')$-path in $T$. Then $\beta'(T'_t) = V(D) - \beta(T_s)$ and $V(D) - \beta'(T'_t) = \beta(T_s)$.
	Since $\mathcal{T}$ satisfies \ref{NCWE_2}, there is no closed walk in $D-\gamma(e')$ containing a vertex of $\beta(T_s)$ and a vertex of $V(D) - \beta(T_s)$. As $\gamma'(e) = \gamma(e')$, there is no closed walk in $D-\gamma'(e)$ containing a vertex of $V(D)- \beta'(T'_t)$ and a vertex of $\beta'(T'_t)$, i.e.~$\mathcal{T'}$ satisfies \ref{NCWE_2} as well.
\end{proof}

As a side remark, the above \namecref{lem:lemmaDifferentRoot} holds only for $\dtwNCW{}$- and \dtdNCWEs.
This is because every other decomposition $(T, \beta, \gamma)$ requires $\beta(T_t) \subseteq V(D) - \gamma(e)$ to be satisfied for every $e = (s,t) \in E(T)$.
The following \namecref{lem:closureLemma1} shows that the \tdtwNCWE of a digraph is not increased by deleting vertices or edges.

\begin{lemma} \label{lem:closureLemma1}
	Let $D$, $H$ be digraphs such that $H \subseteq D$. Then $\dtwNCWE{H} \leq \dtwNCWE{D}$.
\end{lemma}

\begin{proof}
	Let $k$ be the \tdtwNCWE of $D$ and $\mathcal{T} \coloneqq (T, \beta, \gamma)$ be an \dtdNCWE\ of $D$ of width $k$. Then we let $\mathcal{T'} \coloneqq (T, \beta', \gamma')$, where $\beta'(t) \coloneqq \beta(t) \cap V(H)$ for all $t \in V(T)$ and $\gamma'(e) \coloneqq \gamma(e) \cap V(H)$ for all $e \in E(T)$. If $\beta'(r)$ is empty, where $r$ is the root of $T$, then by \cref{lem:lemmaDifferentRoot}, we can find $\mathcal{T'}$ with a non-empty root. Then $\mathcal{T'}$ is an \dtdNCWE\ of $H$ of width at most $k$. 
\end{proof}

Using this, we can proceed to prove our main result: the closeness of \tdtwNCWE under taking butterfly minors.

\closedNCWE*
\begin{proof}
	As $D' \preccurlyeq_b D$, there is a subgraph $H \subseteq D$ such that $D'$ is obtained from $H$ by butterfly contractions only.
    By \cref{lem:closureLemma1}, we have $\dtwNCWE{H} \leq \dtwNCWE{D}$. 
	Let us call the \emph{complexity} of $H$ the number of edges in $H$ that are butterfly contracted to form $D'$.
    We prove this by induction on the complexity of $H$.
    If the complexity is $0$, there is nothing to show. Therefore, assume that the complexity is at least $1$. We choose a butterfly contractable edge $e=(u,v) \in E(H)$ such that $e$ is butterfly contracted in $H$ to obtain $D'$. Let $\hat{D}$ be the digraph obtained from $H$ by butterfly contracting $e$ into the vertex $x_e \in V(\hat{D})$.
	
	Let $k$ be the \tdtwNCWE\ of $H$ and $\mathcal{T} \coloneqq (T, \beta, \gamma)$ be an \dtdNCWE\ of $H$ of width $k$. Let $u_T \in V(T)$ and $v_T \in V(T)$ such that $u \in \beta(u_T)$ and $v \in \beta(v_T)$.
	We construct an \dtdNCWE\ $\mathcal{T'} \coloneqq (T', \beta', \gamma')$ of $\hat{D}$ as follows: 
	\begin{itemize}
		\item $T'$ is an isomorphic copy of $T$,
		\item for all $f \in E(T')$, if $u \in \gamma(f)$ or $v \in \gamma(f)$, then let $\gamma'(f) \coloneqq (\gamma(f) - \{u,v\}) \cup \{x_e\}$; otherwise let $\gamma'(f) \coloneqq \gamma(f)$,
		\item for all $t \in V(T') - \{u_T,v_T\}$, let $\beta'(t) \coloneqq \beta(t)$, and
		\item for $u_T, v_T \in V(T')$, 
		if $\Outdeg{H}{u} = 1$, then let $\beta'(u_T) \coloneqq \beta(u_T) - \{u\}$ and $\beta'(v_T) \coloneqq (\beta(v_T) - \{v\}) \cup \{x_e\}$; otherwise let $\beta'(u_T) \coloneqq (\beta(u_T) - \{u\}) \cup \{x_e\}$ and $\beta'(v_T) \coloneqq \beta(v_T) - \{v\}$. 
	\end{itemize} 
	The width of $\mathcal{T'}$ is not increased by this construction, and $\{\beta'(t): t \in V(T')\}$ is a partition of $V(\hat{D})$ into possibly empty sets.
	
	We claim that $\mathcal{T'}$ satisfies \ref{NCWE_2}. Towards a contradiction, suppose there is an edge $f=(s,t) \in E(T')$ such that there is a closed walk $W'$ in $\hat{D} - \gamma'(f)$ containing a vertex of $\beta'(T'_t)$ and a vertex of $V(\hat{D}) - \beta'(T'_t)$.  We first consider the case where $\gamma'(f) = (\gamma(f) - \{u,v\}) \cup \{x_e\}$.
	Since $x_e \notin V(\hat{D}) - \gamma'(f)$ and $u,v \notin V(\hat{D})$, $x_e, u$ and $v$ are not contained in $V(W')$. Therefore, by construction, $W'$ is also in $H - ((\gamma'(f) - \{x_e\}) \cup \{u,v\}) \subseteq H - \gamma(f)$ and contains a vertex of $\beta(T_t)$ and a vertex of $V(H) - \beta(T_t)$, which is a contradiction.
	
	Next, we consider the case where $\gamma'(f) = \gamma(f)$. Then we have $x_e \notin \gamma'(f)$ and $\{u,v\} \cap \gamma(f) = \emptyset$. Furthermore, $V(\hat{D}) - \gamma'(f) = ((V(H) - \gamma(f)) - \{u,v\}) \cup \{x_e\}$. In other words, $\hat{D} - \gamma'(f)$ can be obtained from $H - \gamma(f)$ by butterfly contracting $e$. 
	Moreover, by a similar argument as above, if $x_e \notin V(W')$, then $W'$ is also in $H - \gamma(f)$, which leads to a contradiction.
	Therefore, the only case in which the closed walk $W'$ exists is that $W'$ contains $x_e$. By construction, we have either $x_e \in \beta'(T'_t)$ or $x_e \in V(\hat{D}) - \beta'(T'_t)$. 
	We consider the former case where $x_e \in \beta'(T'_t)$. The latter case is analogous. If $\Outdeg{H}{u} = 1$, then $v \in \beta(T_t)$. Since $u \notin V(\hat{D}) - \beta'(T'_t)$, $W'$ contains a vertex $w \in V(\hat{D}) - \beta'(T'_t)$ such that $w \neq u$. Then by construction, we have $w \in V(H) - \beta(T_t)$. By \cref{lem:walkInMinor}, there is a closed walk $W$ in $H - \gamma(f)$ containing $v \in \beta(T_t)$ and $w \in V(H) - \beta(T_t)$, which is a contradiction. Otherwise, we have $u \in \beta(T_t)$. Since $v \notin V(\hat{D}) - \beta'(T'_t)$, $W'$ contains a vertex $w \in V(\hat{D}) - \beta'(T'_t)$ such that $w \neq v$. Then again, we have $w \in V(H) - \beta(T_t)$, and by \cref{lem:walkInMinor}, there is a closed walk $W$ in $H - \gamma(f)$ containing $u \in \beta(T_t)$ and $w \in V(H) - \beta(T_t)$, a contradiction.
	
	If $u_T$ or $v_T$ is the root in $T'$, and the bag of the root is empty, then by \cref{lem:lemmaDifferentRoot}, we can obtain an \dtdNCWE\ $\mathcal{T'}$ of $\hat{D}$ of the same width with a non-empty root. Then $\mathcal{T'}$ satisfies \ref{NCWE_1}, and $\mathcal{T'}$ is an \dtdNCWE\ of $\hat{D}$ of width at most $k$. Then we have $\dtwNCWE{\hat{D}} \leq \dtwNCWE{H}$. By the induction hypothesis, it holds that $\dtwNCWE{D'} \leq \dtwNCWE{\hat{D}}$, and hence $\dtwNCWE{D'} \leq \dtwNCWE{H}$.
\end{proof}

\section{\texorpdfstring{$\dtwSCE{}$}{SC\_0} is not a strict upper bound on \texorpdfstring{$\dtwNCW{}$}{NCW}}
\label{sec:counterexample_SCE_being_upper_bound_of_NCW}

The main result of this section is \cref{thm:SCE<NCW}, which states that there is a digraph $D$ with $\dtwSCE{D} < \dtwNCW{D}$, i.e.~\tdtwSCE\ cannot be an upper bound on \tdtwNCW. Moreover, we obtain \cref{cor:D1_c&r_game,cor:D1_haven,cor:D1_bramble,cor:D1_k_linked_set}, which show that the exact min-max theorem between the directed tree-width with respect to \dtwNW{}-, \dtwNCW{}- and \dtdSCDs and the cops-and-robber game does not hold; additionally, the exact duality with the obstructions is not possible.
We first introduce a concept called \emph{strategy trees} to facilitate the proof in this section.

\subsection{Strategy trees}
\label{sec:strategy_trees}
We want to define a strategy tree in such a way that the following statement holds: $k$ cops have a winning strategy on a digraph $D$ if and only if there exists a finite strategy tree of $D$ of width $k$ (see \cite{gradel_graph_2011} for the undirected version).
Given that the cops can see where the robber is, cops' strategies depend on the robber's positions. Assume that $(C_i,v_i)$ is the current game position in a play consistent with a cops' winning strategy on a digraph $D$ with the robber space $R_i \subseteq V(D)$. 
Then $v_i$ is in $R_i$, and the cops do not have to consider the vertices that are not available to the robber.
Furthermore, every position $(C_i,v)$ for $v \in R_i$ is equivalent in the sense that wherever the cops' next position $C_{i+1}$ is, the robber can reach the next robber space $R_{i+1}$ from any vertex $v \in R_i$. Consequently, there is no reason for the cops to treat these cases differently, and the cops' next position will be decided considering $R_i$.
With this in mind, we define a strategy tree as follows.

\begin{definition}[Strategy Tree]
	Let $D$ be a digraph. A \emph{strategy tree} of $D$ is a triple $\mathcal{T}_s \coloneqq (T_s, \cops{}, \robber{})$, where $T_s$ is a rooted directed tree whose nodes $t$ are labelled by $\cops{t} \subseteq V(D)$ and whose edges $e \in E(T)$ are labelled by $\robber{e} \subseteq V(D)$ as follows:
	\begin{enumerate}
		\item if $r$ is the root of $T_s$, then for every strong component $C$ of $D - \cops{r}$, there is an outgoing edge $e \coloneqq (r,t) \in E(T_s)$ such that $V(C) \subseteq \robber{e}$, and $V(C) \cap \robber{e'} = \emptyset$ for all the other outgoing edges $e' \neq e \in E(T_s)$ of $r$, 
		\item if $(s,t) \in E(T_s)$ and $C$ is a strong component of $D - \cops{s}$ with $V(C) \subseteq \robber{(s,t)}$, then
        for each strong component $C'$ of $D - \cops{t}$ contained in the same strong component of $D - (\cops{s} \cap \cops{t})$ as $C$, there is an outgoing edge $e \coloneqq (t,t') \in E(T_s)$ such that $V(C') \subseteq \robber{e}$, and $V(C') \cap \robber{e'} = \emptyset$ for all the other outgoing edges $e' \neq e \in E(T_s)$ of $t$.
	\end{enumerate}
	The \emph{width} of $\mathcal{T}_s$ is defined as max$\{|\cops{t}|: t \in V(T_s)\}$. 
	$\mathcal{T}_s$ is \emph{robber-monotone} if $C \supseteq C'$ holds for every $(s,t), (t,t') \in E(T_s)$ with every strong component $C$ of $D - \cops{s}$ and $C'$ of $D - \cops{t}$ such that $V(C) \subseteq \robber{(s,t)}$ and $V(C') \subseteq \robber{(t,t')}$, and $C$ and $C'$ are contained in the same strong component of $D - (\cops{s} \cap \cops{t})$.
\end{definition}

Each node $t \in V(T_s)$ corresponds to a cop position $\cops{t}$, and each edge $e \in E(T_s)$ corresponds to a robber space or possibly the union of some robber spaces $\robber{e}$, that is, the union of strong components of $D - \cops{t}$.
The conditions of a strategy tree ensure that $\robber{(t,t_i)} \cap \robber{(t,t_j)} = \emptyset$ for all $t \in V(T_s)$ with children $t_1, \dots, t_k \in V(T_s)$ and $1 \leq i < j \leq k$.

The proof for the following \namecref{lem:lemma_dtd_to_strategy_tree} is analogous to the proof of the first part of \cref{thm:c&r_game} shown in \cite{johnson_directed_2001}, so we do not include it here.

\begin{lemma}\label{lem:lemma_dtd_to_strategy_tree}
	Let $D$ be a digraph. If there is a directed tree decomposition $\mathcal{T} \coloneqq (T, \beta, \gamma)$ of $D$ of width $k$ ($\mathcal{T}$ may be any kind of directed tree decomposition from \cref{def:1dtd,def:2dtd,NCWE,def:3dtd,def:4dtd}), then there is a finite strategy tree $\mathcal{T}_s \coloneqq (T_s, \cops{}, \robber{})$ of $D$ of width $k+1$ satisfying
	\begin{itemize}
		\item $T_s$ is an isomorphic copy of $T$,
		\item $\cops{t} = \Gamma(t)$ for $t \in V(T_s)$, and
		\item $\robber{e} = \beta(T_t) - \gamma(e)$ for $e = (s,t) \in E(T_s)$.
	\end{itemize}
\end{lemma}

If $\mathcal{T}$ is an \dtwNW{}-, \dtwSCE{}- or \dtdSCD, then $\beta(T_t) \cap \gamma(e) = \emptyset$ for every $e = (s,t) \in E(T)$, and consequently, the last condition is equivalent to $\robber{e} = \beta(T_t)$ in these cases.
The following \namecref{lemma:strategy_tree_different_root} shows that if a finite strategy tree $\mathcal{T}_s$ of a digraph $D$ is given, we can choose any of its nodes as a root and obtain another finite strategy tree $\mathcal{T}_s'$ of the same width. In other words, any of the cop positions in $\mathcal{T}_s$ can be the cops' start position in their winning strategy.

\begin{lemma} \label{lemma:strategy_tree_different_root}
	Let $D$ be a digraph and $\mathcal{T}_s \coloneqq (T_s, \cops{}, \robber{})$ be a finite strategy tree of $D$ of width $k$. Let $r \in V(T_s)$ be the root of $T_s$ and $r' \in V(T_s)$ such that $r' \neq r$. Then there is another finite strategy tree $\mathcal{T}_s' \coloneqq (T_s', \copsP{}, \robberP{})$ of $D$ of width $k$ such that $r'$ is the root of $T_s'$ and $|V(T_s')| \leq |V(T_s)|$. Furthermore, if $\mathcal{T}_s$ is robber-monotone, then $\mathcal{T}_s'$ is robber-monotone as well.
\end{lemma}

\begin{proof}
	Let $P$ be the $(r,r')$-path in $T_s$. We use induction on the length of $P$. If the length is $0$, then there is nothing to show. Therefore, assume that the length is at least $1$. Let $r^*$ be the predecessor of $r'$ in $T_s$. By the induction hypothesis, there is a finite strategy tree $\mathcal{T}_s^* \coloneqq (T_s^*, \copsS{}, \robberS{})$ of $D$ of width $k$ such that $r^*$ is the root of $T_s^*$ and $|V(T_s^*)| \leq |V(T_s)|$. Furthermore, if $\mathcal{T}_s$ is robber-monotone, then $\mathcal{T}_s^*$ is robber-monotone as well.
	
	We construct $\mathcal{T}_s' \coloneqq (T_s', \copsP{}, \robberP{})$ as follows.
	For all $v \in V(T_s^*)$, let $v \in V(T_s')$ and $\copsP{v} \coloneqq \copsS{v}$, and build $T_s'$ by orienting the edge $(r^*,r') \in E(T_s^*)$ away from $r'$, i.e.~let $(r',r^*) \in E(T_s')$ and for all $e \in E(T_s^*) - \{(r^*,r')\}$, let $e \in E(T_s')$.
    Then $r'$ is the root of $T_s'$. Note that $\{\{(r',r^*)\}, \bigcup \{(t,t') \in E(T_s^*): t$ is reachable from $r'$ in $T_s^*\}, \bigcup \{(t,t') \in E(T_s'): t$ is reachable from $r^*$ in $T_s'\}\}$ is a partition of $E(T_s')$.
	
	Following a breadth-first search of $T_s'$, we assign robber spaces to each edge and possibly delete some nodes and edges. Let $t' \in V(T_s')$ be the current vertex with $t' \neq r'$, $t \in V(T_s')$ be the predecessor of $t'$ and $e \coloneqq (t,t') \in V(T_s')$. We define $\robberP{e}$ as follows:
	
	\begin{itemize}
		
		\item if $t = r'$ and $t' = r^*$, let $\robberP{e} \coloneqq V(D) - (\copsP{r'} \cup \bigcup \{\robberS{e'} : e' = (r',v) \in E(T_s^*)\})$,
		\item if $t$ is reachable from $r'$ in $T_s^*$, let $\robberP{e} \coloneqq \robberS{e}$,

		\item otherwise, let $\robberP{e}$ be the union of vertex sets of strong components $C'$ of $D - \copsP{t}$ such that $V(C') \subseteq \robberS{e}$, and $C'$ is contained in the same strong component of $D - (\copsP{s} \cap \copsP{t})$ as $C$, where $s$ is the predecessor of $t$, and $C$ is a strong component of $D - \copsP{s}$ with $V(C) \subseteq \robberP{(s,t)}$. If $\robberP{e} = \emptyset$, delete $e, t'$, and all nodes reachable from $t'$ and their incident edges in $T_s'$.
	\end{itemize}
	By construction, the width of $\mathcal{T}_s'$ is $k$, and $|V(T_s')| \leq |V(T_s^*)|$, which implies $|V(T_s')| \leq |V(T_s)|$. Furthermore, if $\mathcal{T}_s^*$ is robber-monotone, $\mathcal{T}_s'$ is also robber-monotone.
	Since $\mathcal{T}_s^*$ is finite and $V(T_s') \subseteq V(T_s^*)$, $\mathcal{T}_s'$ is also finite. Therefore, if $\mathcal{T}_s'$ satisfies every condition of a strategy tree, then $\mathcal{T}_s'$ is the desired strategy tree.
	
	We first check whether $\mathcal{T}_s'$ satisfies the first condition of a strategy tree. Since $\mathcal{T}_s^*$ is a strategy tree, for every edge $e \coloneqq (r',v) \in E(T_s^*)$, $\robberS{e}$ is the union of vertex sets of some strong components of $D - \copsS{r'}$. 
	Furthermore, we know that $\robberP{e} = \robberS{e}$ for such an edge $e$ since $e$ is also in $E(T_s')$, and $r'$ is reachable from $r'$ in $T_s^*$. Hence, if $C$ is a strong component of $D - \copsP{r'}$ and $V(C) \subseteq \robberP{e}$, then $V(C) \cap \robberP{e'} = \emptyset$, where $e'$ is an outgoing edge of $r'$ in $T_s^*$ with $e' \neq e$ and $e' \neq (r',r^*)$.
	Moreover, if $C$ is a strong component of $D - \copsP{r'}$, then either
	$V(C) \subseteq \bigcup \{\robberS{e'} : e' = (r',v) \in E(T_s^*)\} = \bigcup \{\robberP{e'} : e' = (r',v) \in E(T_s^*)\}$ or $V(C) \subseteq \robberP{(r',r^*)}$. Clearly, $\robberP{(r',r^*)}$ is disjoint from $\bigcup \{\robberS{e'} : e' = (r',v) \in E(T_s^*)\}$.
	Therefore, for every strong component $C$ of $D - \copsP{r'}$ there is an outgoing edge $e = (r',t) \in E(T_s')$ such that $V(C) \subseteq \robberP{e}$, and for all the other outgoing edges $e' \neq e \in E(T_s')$ of $r'$, we have $V(C) \cap \robberP{e'} = \emptyset$. 
	
	To verify the second condition of a strategy tree, we let $t,s \in V(T_s')$ such that $t \neq r'$ and $s$ is the predecessor of $t$ in $T_s'$, $t_1, \dots, t_k \in V(T_s')$ be the children of $t$ in $T_s'$, and $C$ be a strong component of $D - \copsP{s}$ with $V(C) \subseteq \robberP{(s,t)}$. Then we know $C$ is a strong component of $D - \copsS{s}$. First, we consider the case where $t$ is reachable from $r'$ in $T_s^*$. Then we have $\robberP{e} = \robberS{e}$ for every edge $e \in V(T_s')$ that lies on the $(r',t_i)$-path for every $1 \leq i \leq k$. As $\mathcal{T}_s^*$ is a strategy tree, the second condition holds trivially in this case. 
	
	Second, we consider the case where $t = r^*$ and $s = r'$. Towards a contradiction, suppose there is a strong component $C'$ of $D - \copsP{r^*}$ contained in the same strong component of $D - (\copsP{r'} \cap \copsP{r^*})$ as $C$ with $V(C') \nsubseteq \robberP{(r^*,t_i)}$ for every $1 \leq i \leq k$. Then by construction, $V(C') \nsubseteq \robberS{(r^*,t_i)}$ for every $1 \leq i \leq k$, and hence $V(C') \subseteq \robberS{(r^*,r')}$. Since $C$ and $C'$ lie in the same strong component of $D - (\copsP{r'} \cap \copsP{r^*}) = D - (\copsS{r'} \cap \copsS{r^*})$, and $C$ is a strong component of $D - \copsS{r'}$, there is an edge $e' = (r',v) \in E(T_s^*)$ such that $V(C) \subseteq \robberS{e'}$, which is a contradiction to that $V(C) \subseteq \robberP{(r',r^*)} = V(D) - (\copsP{r'} \cup \bigcup \{\robberS{e'} : e' = (r',v) \in E(T_s^*)\}$. Therefore, for every such strong component $C'$ of $D - \copsP{r^*}$, there is at least one outgoing edge $(r^*,t_i) \in E(T_s')$ such that $V(C') \subseteq \robberP{(r^*,t_i)}$.
	Suppose there are two outgoing edges $(r^*,t_i), (r^*,t_j) \in E(T_s')$ with $V(C') \subseteq \robberP{(r^*,t_i)}$ and $V(C') \cap \robberP{(r^*,t_j)} \neq \emptyset$ for $1 \leq i \neq j \leq k$. By construction, we know $\robberP{e} \subseteq \robberS{e}$ for every $e \in E(T_s') - \{(r',r^*)\}$. Then $V(C') \subseteq \robberS{(r^*,t_i)}$ and $V(C') \cap \robberS{(r^*,t_j)}$ $\neq \emptyset$, a contradiction.
	
	Finally, we consider every remaining node $t$, i.e.~every node $t$ that is reachable from $r^*$ in $T_s'$ with $t \neq r^*$. Suppose there is a strong component $C'$ of $D - \copsP{t}$ contained in the same strong component of $D - (\copsP{s} \cap \copsP{t})$ as $C$ with $V(C') \nsubseteq \robberP{(t,t_i)}$ for every $1 \leq i \leq k$. 
	If $V(C') \subseteq \robberS{(t,t_i)}$ for any $i$, then by construction $V(C')$ must be contained in $\robberP{(t,t_i)}$.
	Therefore, we have $V(C') \nsubseteq \robberS{(t,t_i)}$ for every $1 \leq i \leq k$. Since $V(C) \subseteq \robberP{(s,t)} \subseteq \robberS{(s,t)}$ (because $(s,t) \in E(T_s') - \{(r',r^*)\}$), this is a contradiction to that $\mathcal{T}_s^*$ satisfies the second condition of a strategy tree.
	Hence, for every such strong component $C'$ of $D - \copsP{t}$ there is at least one outgoing edge $(t,t_i) \in E(T_s')$ such that $V(C') \subseteq \robberP{(t,t_i)}$ for $1 \leq i \leq k$. By the same argument as in the last case, there is exactly one such edge $(t,t_i) \in E(T_s')$, and $\robberP{e'}$ of every other outgoing edge $e'$ of $t$ is disjoint from $V(C')$.
	Hence, $\mathcal{T}_s'$ satisfies every condition of a strategy tree.
\end{proof}

\subsection{The \texorpdfstring{$\dtwSCE{}$}{SC\_0} can be strictly less than the \texorpdfstring{\tdtwNCW}{NCW-directed \treewidth}}
\label{sec:SC0_no_upper_bound_on_NCW}

Now, we are equipped with the tools needed to study the digraph $D_1$ from~\cref{fig:D1} and its properties closer.

\begin{lemma} \label{lem:robber-mon_strategy_for_D_1}
	Let $D_1$ be the digraph depicted in~\cref{fig:D1}. There is a robber-monotone winning strategy for $4$ cops in the cops-and-robber game on $D_1$. 
\end{lemma}

\begin{sidewaysfigure}
	\centering
	\begin{tikzpicture}
		[
			state/.style={circle, draw=black, thick, font=\small,inner sep=0,outer sep=0, minimum size=6mm},
			blank/.style={inner sep=0,outer sep=0},
			> = latex, 
			shorten > = 1pt, 
			auto,
			node distance = 1.5cm, 
			line width= 0.3mm,
			scale=.67,
			transform shape
			]
		
		\node[state] (0) {$0$};
		\node[state] (0') [below of=0] {$0'$};
		
		\node[state] (a) [right of=0]{$a$};
		\node[state] (a') [below of=a] {$a'$};
		\node[state] (b) [right of=a] {$b$};
		\node[state] (b') [below of=b] {$b'$};
		\node[state] (c) [right of=b] {$c$};
		\node[state] (c') [below of=c] {$c'$};
		\node[state] (d) [right of=c] {$d$};
		\node[state] (d') [below of=d] {$d'$};
		
		\node[state] (-a) [left of=0]{$-a$};
		\node[state] (-a') [below of=-a] {$-a'$};
		\node[state] (-b) [left of=-a] {$-b$};
		\node[state] (-b') [below of=-b] {$-b'$};
		\node[state] (-c) [left of=-b] {$-c$};
		\node[state] (-c') [below of=-c] {$-c'$};
		\node[state] (-d) [left of=-c] {$-d$};
		\node[state] (-d') [below of=-d] {$-d'$};
		
		\node[state] (1) [right of=d] {$1$};
		\node[state] (1') [below of=1] {$1'$};
		\node[state] (2) [right of=1] {$2$};
		\node[state] (2') [below of=2] {$2'$};
		\node[state] (3) [right of=2] {$3$};
		\node[state] (3') [below of=3] {$3'$};
		\node[state] (4) [right of=3] {$4$};
		\node[state] (4') [below of=4] {$4'$};
		
		\node[state] (-1) [left of=-d] {$-1$};
		\node[state] (-1') [below of=-1] {$-1'$};
		\node[state] (-2) [left of=-1] {$-2$};
		\node[state] (-2') [below of=-2] {$-2'$};
		\node[state] (-3) [left of=-2] {$-3$};
		\node[state] (-3') [below of=-3] {$-3'$};
		\node[state] (-4) [left of=-3] {$-4$};
		\node[state] (-4') [below of=-4] {$-4'$};

		\path[-] (0) edge node {} (0');
		\path[-] (0) edge node {} (a);
		\path[-] (0') edge node {} (a');
		\path[-] (0') edge node {} (a);
		\path[-] (0) edge node {} (a');
		
		\path[-] (a) edge node {} (a');
		\path[-] (a) edge node {} (b);
		\path[-] (a) edge node {} (b');
		\path[-] (a') edge node {} (b');
		\path[-] (a') edge node {} (b);
		
		\path[-] (b) edge node {} (b');
		\path[-] (b) edge node {} (c);
		\path[-] (b) edge node {} (c');
		\path[-] (b') edge node {} (c');
		
		\path[-] (c) edge node {} (c');
		\path[-] (c) edge node {} (d);
		\path[-] (c) edge node {} (d');
		\path[-] (c') edge node {} (d');
		\path[-] (c') edge node {} (d);
		
		\path[-] (d) edge node {} (d');
		\path[-] (d) edge node {} (1);
		\path[-] (d) edge node {} (1');
		\path[-] (d') edge node {} (1');

		\path[-] (1) edge node {} (1');
		\path[-] (1) edge node {} (2);
		\path[-] (1) edge node {} (1');
		\path[-] (1) edge node {} (2');
		\path[-] (1') edge node {} (2');
		\path[-] (1') edge node {} (2);
		\path[-] (2) edge node {} (2');
		\path[-] (2) edge node {} (3);
		\path[-] (2) edge node {} (2');
		\path[-] (2') edge node {} (3');
		\path[-] (3) edge node {} (3');
		\path[-] (2) edge node {} (3');
		\path[-] (3) edge node {} (4);
		\path[-] (3) edge node {} (4');
		\path[-] (3') edge node {} (4');
		\path[-] (4') edge node {} (4);

		\path[-] (0) edge node {} (-a);
		\path[-] (0') edge node {} (-a');
		\path[-] (0') edge node {} (-a);
		\path[-] (0) edge node {} (-a');
		
		\path[-] (-a) edge node {} (-a');
		\path[-] (-a) edge node {} (-b);
		\path[-] (-a) edge node {} (-b');
		\path[-] (-a') edge node {} (-b');
		\path[-] (-a') edge node {} (-b);
		
		\path[-] (-b) edge node {} (-b');
		\path[-] (-b) edge node {} (-c);
		\path[-] (-b) edge node {} (-c');
		\path[-] (-b') edge node {} (-c');
		
		\path[-] (-c) edge node {} (-c');
		\path[-] (-c) edge node {} (-d);
		\path[-] (-c) edge node {} (-d');
		\path[-] (-c') edge node {} (-d');
		\path[-] (-c') edge node {} (-d);
		
		\path[-] (-d) edge node {} (-d');
		\path[-] (-d) edge node {} (-1);
		\path[-] (-d) edge node {} (-1');
		\path[-] (-d') edge node {} (-1');

		\path[-] (-1) edge node {} (-1');
		\path[-] (-1) edge node {} (-2);
		\path[-] (-1) edge node {} (-1');
		\path[-] (-1) edge node {} (-2');
		\path[-] (-1') edge node {} (-2');
		\path[-] (-1') edge node {} (-2);
		\path[-] (-2) edge node {} (-2');
		\path[-] (-2) edge node {} (-3);
		\path[-] (-2) edge node {} (-2');
		\path[-] (-2') edge node {} (-3');
		\path[-] (-3) edge node {} (-3');
		\path[-] (-2) edge node {} (-3');
		\path[-] (-3) edge node {} (-4);
		\path[-] (-3) edge node {} (-4');
		\path[-] (-3') edge node {} (-4');
		\path[-] (-4') edge node {} (-4);
		
		\path[->, red] (4) edge node {} (2');
		\path[->, red] (4') edge node {} (2);
		\path[->, red] (4) edge[bend right=30] node[above] {} (2);
		\path[->, red] (4') edge[bend left=30] node[below] {} (2');
		\draw[->, red] (4) to[out=70, in=30, looseness=1] (1);
		\draw[->, red] (4') to[out=50, in=50, looseness=2] (1);
		\draw[->, blue] (0) to[out=40, in=50, looseness=0.8] (4);
		\draw[->, blue] (0) to[out=50, in=45, looseness=1.4] (4');

		\path[->, red] (1) edge[bend right=30] node[above] {} (c);
		\path[->, red] (c) edge[bend right=30] node[above] {} (a);
		\path[->, red] (b) edge[bend left=30] node[below] {} (0);
		\path[->, red] (2) edge[bend left=25] node[below] {} (d);
		\path[->, red] (2) edge node {} (d');
		\path[->, red] (d) edge[bend left=25] node[below] {} (b);
		\path[->, red] (d) edge node {} (b');
		
		\path[->, red] (-4) edge node {} (-2');
		\path[->, red] (-4') edge node {} (-2);
		\path[->, red] (-4) edge[bend left=30] node[above] {} (-2);
		\path[->, red] (-4') edge[bend right=30] node[below] {} (-2');
		\draw[->, red] (-4) to[out=110, in=150, looseness=1] (-1);
		\draw[->, red] (-4') to[out=130, in=130, looseness=2] (-1);
		\draw[->, blue] (0) to[out=140, in=130, looseness=0.8] (-4);
		\draw[->, blue] (0) to[out=130, in=135, looseness=1.4] (-4');

		\path[->, red] (-1) edge[bend left=30] node[above] {} (-c);
		\path[->, red] (-c) edge[bend left=30] node[above] {} (-a);
		\path[->, red] (-b) edge[bend right=30] node[below] {} (0);
		\path[->, red] (-2) edge[bend right=25] node[below] {} (-d);
		\path[->, red] (-2) edge node {} (-d');
		\path[->, red] (-d) edge[bend right=25] node[below] {} (-b);
		\path[->, red] (-d) edge node {} (-b');
		
	\end{tikzpicture}
	\caption{The digraph $D_1$ from \cref{thm:SCE<NCW} with $\dtwSCE{D_1} < \dtwNCW{D_1}$. The digraph is a modification of the example in \cite[Fig.~4]{adler_directed_2007}.}
	\label{fig:D1}
	\vspace{3cm}
	\begin{tikzpicture}[grow=right ,sibling distance=55pt, level distance=50pt, scale=.65,transform shape, auto,level 1/.style={sibling distance=3cm},level 2/.style={sibling distance=1.5cm}, level 3/.style={sibling distance=1cm}]
		\tikzstyle{edge from parent}=[> = latex, ->, thick, draw]
		\tikzset{tree/.style={draw,rounded corners,thick, fill=blue!20, minimum size=0.6cm, level distance=1.25cm,sibling distance=2cm,inner sep=2pt}}
		\tikzset{rec/.append style={rectangle,draw=black, thin, fill=white, inner sep=1pt}}
		
		\node [tree] {$0,0'$}
		child {node [tree] {$-a,-a'$} 
			child {node [tree] {$-b$}
				child {node [tree] {$-b'$}
					child {node [tree] {$\emptyset$}
						child {node [tree] {$-c'$}
							child {node [tree] {$-c$}
								child {node [tree] {$-d$}
									child {node [tree] {$-d'$}
										child {node [tree] {$\emptyset$}
											child {node [tree] {$-1'$}
												child {node [tree] {$-1$}
													child {node [tree] {$-2$}
														child {node [tree] {$-2'$}
															child {node [tree] {$\emptyset$}
																child {node [tree] {$-3'$}
																	child {node [tree] {$-3$}
																		child {node [tree] {$-4'$}
																			child {node [tree] {$-4$}
																				edge from parent
																				node[rec, below=0.4cm] {$0,-3,-4'$}
																			}
																			edge from parent
																			node[rec, above=0.4cm] {$0,-3,-3'$}
																		}
																		edge from parent
																		node[rec, below=0.4cm] {$0,-2,-3'$}
																	}
																	edge from parent
																	node[rec, above=0.4cm] {$0,-2,-2'$}
																}
																edge from parent
																node[rec, below=0.4cm] {$-1,-2,-2'$}
															}
															edge from parent
															node[rec, above=0.4cm] {$-1,-1',-2$}
														}
														edge from parent
														node[rec, below=0.4cm] {$0,-1,-1'$}
													}
												edge from parent
												node[rec, above=0.4cm] {$0,-d,-1'$}
											}
											edge from parent
											node[rec, below=0.4cm] {$0,-d,-d'$}
										}
										edge from parent
										node[rec, above=0.4cm] {$-c,-d,-d'$}
									}
									edge from parent
									node[rec, below=0.4cm] {$-c,-c',-d$}
								}
								edge from parent
								node[rec, above=0.4cm] {$0,-c,-c'$}
							}
							edge from parent
							node[rec, below=0.4cm] {$0,-b,-c'$}
						}
						edge from parent
						node[rec, above=0.4cm] {$0,-b,-b'$}
					}
					edge from parent
					node[rec, below=0.4cm] {$-a,-b,-b'$}
				}
				edge from parent
				node[rec, above=0.4cm] {$-a,-a',-b$}
			}
			edge from parent
			node[rec, below=0.4cm] {$0,-a,-a'$}
		}
		edge from parent
		node[rec, below=0.3cm] {$0,0'$}
	}
		child {node [tree] {$a,a'$} 
			child {node [tree] {$b$}
				child {node [tree] {$b'$}
					child {node [tree] {$\emptyset$}
						child {node [tree] {$c'$}
							child {node [tree] {$c$}
								child {node [tree] {$d$}
									child {node [tree] {$d'$}
										child {node [tree] {$\emptyset$}
											child {node [tree] {$1'$}
												child {node [tree] {$1$}
													child {node [tree] {$2$}
														child {node [tree] {$2'$}
															child {node [tree] {$\emptyset$}
																child {node [tree] {$3'$}
																	child {node [tree] {$3$}
																		child {node [tree] {$4'$}
																			child {node [tree] {$4$}
																				edge from parent
																				node[rec, above=0.2cm] {$0,3,4'$}
																			}
																			edge from parent
																			node[rec, above=0.2cm] {$0,3,3'$}
																		}
																		edge from parent
																		node[rec, above=0.2cm] {$0,2,3'$}
																	}
																	edge from parent
																	node[rec, above=0.2cm] {$0,2,2'$}
																}
																edge from parent
																node[rec, above=0.2cm] {$1,2,2'$}
															}
															edge from parent
															node[rec, above=0.2cm] {$1,1',2$}
														}
														edge from parent
														node[rec, above=0.2cm] {$0,1,1'$}
													}
													edge from parent
													node[rec, above=0.2cm] {$0,d,1'$}
												}
												edge from parent
												node[rec, above=0.2cm] {$0,d,d'$}
											}
											edge from parent
											node[rec, above=0.2cm] {$c,d,d'$}
										}
										edge from parent
										node[rec, above=0.2cm] {$c,c',d$}
									}
									edge from parent
									node[rec, above=0.2cm] {$0,c,c'$}
								}
								edge from parent
								node[rec, above=0.2cm] {$0,b,c'$}
							}
							edge from parent
							node[rec, above=0.2cm] {$0,b,b'$}
						}
						edge from parent
						node[rec, above=0.2cm] {$a,b,b'$}
					}
					edge from parent
					node[rec, above=0.2cm] {$a,a',b$}
				}
				edge from parent
				node[rec, above=0.2cm] {$0,a,a'$}
			}
			edge from parent
			node[rec, above=0.2cm] {$0,0'$}
		};
		
	\end{tikzpicture}
	\caption{An \dtdSCE\ of $D_1$ of width $3$, implying $\dtwSCE{D_1} \leq 3$.}
	\label{fig:SCE_dtd_D1}
\end{sidewaysfigure}

\begin{proof}
	Here is a robber-monotone winning strategy for $4$ cops: first, the cops are placed on $\{0,0'\}$. Then, the robber is either in the positive part of the graph or in its negative counterpart.
    Without loss of generality, we can assume that the robber is in the positive part.
    Then, the cops occupy $\{0,0',a,a'\}$, $\{0,a,a',b\}$, $\{a,a',b,b'\}$, $\{a,b,b',0\}$, $\{b,b',0,c'\}$, $\{b,0,c',c\}$, $\{0,c',c,d\}$, $\{c',c,d,d'\}$, $\{c,d,d',0\}$, $\{d,d',0,1'\}$, $\{d,0,1',1\}$, $\{0,1',1,2\}$, $\{1',1,$ $2,2'\}$, $\{1,2,2',0\}$, $\{2,2',$ $0,3'\}$, $\{2,0,3',3\}$, $\{0,3',3,4'\}$ and $\{0,3,4',4\}$ consecutively.
\end{proof}

It is noteworthy that the above strategy is not a cop-monotone strategy since the cops return to $0$ repeatedly.

The next \namecref{lem:D_1_NCWgeq4} is the most involved statement to prove in this section.
Its proof shows the intrinsic necessity of empty bags in directed tree decompositions equivalent to the cops-and-robber game.
By the proof of \cref{thm:c&r_game}, the \dtdSCE\ in \cref{fig:SCE_dtd_D1} yields a winning strategy for $4$ cops on $D_1$.
In fact, the strategy corresponds to the one in the proof of \cref{lem:robber-mon_strategy_for_D_1}.
Whenever there is an empty bag in the decomposition, the cops reoccupy $0$ in the corresponding round.
In such rounds, they do not reduce the robber space but change the guards by reoccupying $0$. 

\begin{lemma}
    \label{lem:D_1_NCWgeq4}
    Let $D_1$ be the digraph depicted in~\cref{fig:D1}. Then, $\dtwNCW{D_1} \geq 4$.
\end{lemma}
\begin{proof}
    Towards a contradiction, suppose there is an \dtdNCW\ $\mathcal{T} \coloneqq (T, \beta, \gamma)$ of $D_1$ of width $3$. 
	Then there is a finite strategy tree $\mathcal{T}_s \coloneqq (T_s, \cops{}, \robber{})$ of $D_1$ of width $4$ satisfying the properties in \cref{lem:lemma_dtd_to_strategy_tree}. By definition, $\{\beta(t): t \in V(T)\}$ is a partition of $V(D_1)$ into non-empty sets, and every bag must contain at least one vertex. Then by \cref{lem:lemma_dtd_to_strategy_tree}, $|V(T_s)| = |V(T)| \leq |V(D_1)| = 34$.
	
	We show by a series of claims that $\mathcal{T}_s$ contains at least $36$ nodes, which leads to a contradiction. Assume that $\mathcal{T}_s$ contains the minimum number of nodes. Let $t_1 \in V(T_s)$ be the root of $T_s$. 
	Since there is a robber-monotone strategy for $4$ cops on $D_1$ (by \cref{lem:robber-mon_strategy_for_D_1}), $\mathcal{T}_s$ is also robber-monotone (as otherwise $\mathcal{T}_s$ does not have the minimum number of nodes).
	For every $s \in V(T_s)$, we have $|\cops{s}| \leq 4$ as the width of $\mathcal{T}_s$ is $4$. 
	Since every finite strategy tree corresponds to a winning strategy for cops, we consider the play consistent with the strategy given by $\mathcal{T}_s$ in the following claims.
	Let $\Delta \coloneqq \{b,b',c,c',d,d',1,1',2,2',3,3',4,4'\}$ and $-\Delta \coloneqq \{-b,-b',-c,-c',-d,-d',-1,-1',-2,-2',$ $-3,-3',-4,-4'\}$.
	
	\begin{claim} \label{claim:st1}
		$T_s$ contains a node $t \in V(T_s)$ such that $\cops{t} = \{0,0',a,a'\}$.
	\end{claim}
	
	\begin{proof*}
		We claim that every finite strategy must contain such a node by giving a winning strategy for the robber against $4$ cops who do not occupy $0,0',a$ and $a'$ simultaneously. The winning strategy is as follows: The robber starts the game by occupying one of $\{0,0',a,a'\}$ and stays there until the robber's current position is included in the cops' next move. As the cops do not occupy $0,0',a$ and $a'$ simultaneously, and $D_1[\{0,0',a,a'\}]$ is a clique of size $4$, there is always at least one vertex $v \in \{0,0',a,a'\}$ to which the robber can escape. In this way, the robber can elude cops continuously.
	\end{proof*}
	
	By \cref{lemma:strategy_tree_different_root,claim:st1}, we may assume that the root $t_1$ has $\cops{t_1} = \{0,0',a,a'\}$ without contradicting the minimality and the monotonicity of $\mathcal{T}_s$.

	\begin{claim} \label{claim:st2}
		$T_s$ has at most two leaves $l,l' \in V(T_s)$. Furthermore, $\cops{l} = \{0,3,4',4\}$ and $\cops{l'} = \{0,-3,-4',-4\}$.
	\end{claim}
	
	\begin{proof*}
		Let $l \in V(T_s)$ be a leaf. We have $\robber{(\pred{l},l)} \neq \emptyset$, as otherwise we can delete $l$, and $T_s$ does not contain the minimum number of nodes. Since $T_s$ is finite, and $l$ is a leaf of $T_s$, the cops can catch the robber by moving from $\cops{\pred{l}}$ to $\cops{l}$, and the robber has nowhere to escape while the cops are moving.
		In terms of the strategy tree, if $C$ is a strong component of $D_1 - \cops{\pred{l}}$ with $V(C) \subseteq \robber{(\pred{l},l)}$, there is no strong component $C'$ of $D_1 - \cops{l}$ contained in the same strong component of $D_1 - (\cops{\pred{l}} \cap \cops{l})$ as $C$.
		Let $R$ be the strong component of $D_1 - (\cops{\pred{l}} \cap \cops{l})$ that contains $C$. Then we have $R \subseteq \cops{l} - (\cops{\pred{l}} \cap \cops{l})$. As $C \neq \emptyset$, $R$ contains at least one vertex.
		Note that $|\cops{\pred{l}} \cap \cops{l}| \leq |\cops{l}| - |R|$ and $|\cops{l}| \leq 4$.
		
		If $|R| = 1$, then $|\cops{\pred{l}} \cap \cops{l}| \leq 3$, i.e.~at most $3$ cops can remain in $D_1$ to guard $R$ of size $1$. Every $v \in V(D_1) - \{4,4',-4',-4'\}$ has at least four neighbours connected by an undirected edge. Therefore, no such $v$ can be contained in $R$. Furthermore, if $V(R) = \{4'\}$, then $\cops{\pred{l}} \cap \cops{l}$ must contain more than $3$ vertices since there are closed walks starting from $4'$ that do not contain the three neighbours $\{3,3',4\}$ of $4'$ connected by an undirected edge. Therefore, neither $4'$ nor $-4'$ can be contained in $R$ (due to symmetry). Then the possible choices for $R$ are $V(R) = \{4\}$ with $\cops{\pred{l}} \cap \cops{l} = \{0,3,4'\}$ and the negative counterpart (due to symmetry). Hence, the cop position for a leaf can be $\{0,3,4',4\}$ or $\{0,-3,-4',-4\}$.
		Otherwise, we have $2 \leq |R| \leq 4$. Then $0 \leq |\cops{\pred{l}} \cap \cops{l}| \leq 2$; i.e.~at most $2$ cops can remain in $D_1$ to guard $R$ of size $2$, $1$ cop for $R$ of size $3$ and $0$ cops for $R$ of size $4$. It is straightforward to verify that this is not possible.
		Hence, $T_s$ has at most two leaves $l,l' \in V(T_s)$. Furthermore, $\cops{l} = \{0,3,4',4\}$ and $\cops{l'} = \{0,-3,-4',-4\}$.
	\end{proof*}

    As every outgoing edge has at least one leaf, \cref{claim:st2} implies the following claim.
	
	\begin{claim} \label{claim:st3}
		$T_s$ has at most one node of out-degree $2$, and every other node has out-degree at most $1$.
	\end{claim}
	
	\begin{claim} \label{claim:st4}
		The root $t_1$ has out-degree $2$ in $T_s$.
	\end{claim}	
	
	\begin{proof*}
		Due to \cref{claim:st3}, $t_1$ cannot have out-degree greater than $2$.
		Suppose $t_1$ has out-degree $1$ in $T_s$. Let $s \in V(T_s)$ be the child of $t_1$. Then $\robber{(t_1,s)}$ is the union of the vertex sets of the strong components $C$ and $C'$ of $D_1 - \cops{t_1}$ with $V(C) = \Delta$ and $V(C') = -\Delta \cup \{-a,-a'\}$, i.e.~$\robber{(t_1,s)} = V(C) \cup V(C')$. Assume the cops are placed on $\cops{t_1} = \{0,0',a,a'\}$. Since each vertex of $\{0,0',a,a'\}$ has two neighbours in either $C$ or $C'$ connected by an undirected edge, once one of the cops leaves his position, there is at least one strong component from which the robber can move to the released position. Suppose the cops never occupy $\{0,0',a,a'\}$ in the later rounds after they leave this position. Then the robber wins against $4$ cops just by staying in $D_1[\{0,0',a,a'\}]$ (with the same strategy presented in \cref{claim:st1}). Therefore, the cops have to occupy $\{0,0',a,a'\}$ again after leaving $\cops{t_1}$, which means there are at least two nodes in $\mathcal{T}_s$ with the same cop position $\{0,0',a,a'\}$. This contradicts the minimality of $\mathcal{T}_s$ since there is a robber-monotone winning strategy for $4$ cops, where the cops occupy $\{0,0',a,a'\}$ only once, given by the proof of \cref{lem:robber-mon_strategy_for_D_1}.
	\end{proof*}

    From \cref{claim:st2,claim:st3,claim:st4} we obtain the following claim.

	\begin{claim} \label{claim:st5}
		$T_s$ has exactly two leaves $l,l' \in V(T_s)$. Furthermore, $\cops{l} = \{0,3,4',4\}$ and $\cops{l'} = \{0,-3,-4',-4\}$.
	\end{claim}
	
	By \cref{claim:st3,claim:st4}, each node in $T_s$ except for the root and the two leaves has precisely one child. Since the root $t_1$ has two outgoing edges, and there are two strong components $C$ and $C'$ of $D_1 - \cops{t_1}$ such that $V(C) = \Delta$ and $V(C') = -\Delta \cup \{-a,-a'\}$, one of the outgoing edges has the robber space $V(C)$ and the other has $V(C')$.

	\begin{claim} \label{claim:st6}
		Let $-t_1,t_2 \in V(T_s)$ be the children of $t_1$ such that $\robber{(t_1,-t_1)} = -\Delta \cup \{-a,-a'\}$ and $\robber{(t_1,t_2)} = \Delta$. Then $\cops{-t_1} = \{0,0',-a,-a'\}$ and $\cops{t_2} = \{0,a,a',b\}$.
	\end{claim}
	
	\begin{proof*}
		Assume the cops are placed on $\cops{t_1} = \{0,0',a,a'\}$, and the robber is in the space $\robber{(t_1,-t_1)}$ $= -\Delta \cup \{-a,-a'\}$. If one of the cops in $\{0,0'\}$ moves, then the robber space increases (since $-a$ and $-a'$ are the neighbours of $0$ and $0'$ connected by an undirected edge). However, the cops in $\{a,a'\}$ can move without increasing the robber space. Since $D_1[\{0,0',-a,-a'\}]$ is a clique of size $4$, the cops have to occupy $\{0,0',-a,-a'\}$ in their next move to obtain a robber-monotone strategy tree with the minimum number of nodes, i.e.~$\cops{-t_1} = \{0,0',-a,-a'\}$.
		
		Now, assume that the robber is in $\robber{(t_1,t_2)} = \Delta$. Then, no cop in $\{0,a,a'\}$ can move without increasing the robber space, but the cop in $0'$ can. Suppose the cop moves to a vertex $v \in V(D_1) - \{b\}$. Then, in the next round, no cop in $\{0,a,a'\}$ can move without increasing the robber space. Furthermore, immediately leaving $v$ right after occupying it contradicts the minimality of $\mathcal{T}_s$. If the released cop moves to $b$, then the cop in $0$ can move without increasing the robber space, and we have the remaining robber space $\Delta - \{b\}$. Therefore, we have $\cops{t_2} = \{0,a,a',b\}$.
	\end{proof*}

	\begin{claim} \label{claim:st7}
		Let $t_3 \in V(T_s)$ be the child of $t_2$ with $\robber{(t_2,t_3)} = \Delta - \{b\}$. Then $\cops{t_3} = \{a,a',b,b'\}$.
	\end{claim}
	
	\begin{proof*}
		By a similar argument as in \cref{claim:st6}.
	\end{proof*}

	\begin{claim} \label{claim:st8}
		Let $t_4 \in V(T_s)$ be the child of $t_3$ with $\robber{(t_3,t_4)} = \Delta - \{b,b'\}$. Then $\cops{t_4} = \{a,b,b',0\}$.
	\end{claim}
	
	\begin{proof*}
		Assume the cops are placed on $\cops{t_3} = \{a,a',b,b'\}$, and the robber is in $\Delta - \{b,b'\}$. Then, no cop in $\{a,b,b'\}$ can move without increasing the robber space, but the cop in $a'$ can. 
		By a similar argument as in \cref{claim:st6}, the cop cannot occupy a vertex $v \in V(D_1) - \{c,0\}$.
		Therefore, we have two choices for the next move, namely $\{a,b,b',c\}$ and $\{a,b,b',0\}$.
		If they move to $\{a,b,b',c\}$, the cop in $a$ can move without increasing the robber space, while the others cannot. Then, whichever the cop chooses to occupy, no cop in $\{b,b',c\}$ can move without increasing the robber space.
		However, if the cops move to $\{a,b,b',0\}$, then the cop in $a$ can move, and they can occupy $\{b,b',0,c'\}$, where the cop in $b'$ can move without increasing the robber space. 
		Due to the minimality of $\mathcal{T}_s$, we have $\cops{t_4} = \{a,b,b',0\}$.
	\end{proof*}

	As the remaining robber space $\Delta - \{b,b'\}$ has a similar pattern, the proof for the following claim resembles the proofs of the claims above.
    The robber space for each edge is omitted since it is clear from the context.
	
	\begin{claim} \label{claim:st9}
		Let $t_{i+1} \in V(T_s)$ be the child of $t_i$ for $i \in \{4, ..., 17\}$. Then we have $\cops{t_5} = \{b,b',0,c'\}$, $\cops{t_6} = \{b,0,c',c\}$, $\cops{t_7} = \{0,c',c,d\}$, $\cops{t_8} = \{c',c,d,d'\}$, $\cops{t_9} = \{c,d,d',0\}$, $\cops{t_{10}} = \{d,d',0,1'\}$, $\cops{t_{11}}$ $= \{d,0,1',1\}$, $\cops{t_{12}} = \{0,1',1,2\}$, $\cops{t_{13}}$ $= \{1',1,2,2'\}$, $\cops{t_{14}} = \{1,2,2',0\}$, $\cops{t_{15}} = \{2,2',0,3'\}$, $\cops{t_{16}} = \{2,0,3',3\}$, $\cops{t_{17}} = \{0,3',3,4'\}$ and $\cops{t_{18}} = \{0,3,4',4\}$.
	\end{claim}
	
	By \cref{claim:st1,claim:st6,claim:st7,claim:st8,claim:st9} and due to symmetry, $\mathcal{T}_s$ contains $36$ nodes. As $\mathcal{T}_s$ contains the minimum number of nodes, every finite strategy tree $\mathcal{T}_s$ must contain at least $36$ nodes.
\end{proof}

This establishes that $D_1$ witnesses \tdtwSCE not being an upper bound on \tdtwNCW.

\CompSCENCW*

\begin{proof}
    Let $D_1$ be the digraph depicted in \cref{fig:D1}.
    We prove that indeed $\dtwSCE{D_1} < \dtwNCW{D_1}$.
	The \dtdSCE\ in \cref{fig:SCE_dtd_D1} shows that $\dtwSCE{D_1} \leq 3$ and, by \cref{lem:D_1_NCWgeq4}, we have $\dtwNCW{D_1} \geq 4$.
\end{proof}

We obtain a few more insights on $D_1$ from our observations.

\begin{corollary} \label{cor:D1_c&r_game}
	Let $D_1$ be the digraph depicted in \cref{fig:D1}. Then $4$ cops have a winning strategy in the cops-and-robber game on $D_1$. Furthermore, it holds that $\dtwNW{D_1} \geq 4$, $\dtwNCW{D_1} \geq 4$ and $\dtwSCD{D_1} \geq 4$.
\end{corollary}
\begin{proof}
	Due to the first part of \cref{thm:c&r_game} and the \dtdSCE\ in \cref{fig:SCE_dtd_D1}, $4$ cops have a winning strategy on $D_1$. By the proof of \cref{thm:SCE<NCW} and due to the relation shown in \cref{fig:figure_dtw_compare}, we have $\dtwNW{D_1} \geq 4$, $\dtwNCW{D_1} \geq 4$ and $\dtwSCD{D_1} \geq 4$.
\end{proof}
The above corollary shows that the converse of the first part of \cref{thm:c&r_game} does not hold for \dtwNW{}-, \dtwNCW{}- and \tdtwSCD.
\begin{corollary} \label{cor:D1_haven}
	Let $D_1$ be the digraph depicted in \cref{fig:D1}. Then $D_1$ has no haven of order $5$. Furthermore, it holds that $\dtwNW{D_1} \geq 4$, $\dtwNCW{D_1} \geq 4$ and $\dtwSCD{D_1} \geq 4$.
\end{corollary}
\begin{proof}
	By~\cref{cor:D1_c&r_game} and the first part of \cref{thm:theorem_haven}.
\end{proof}

The above corollary shows that the converse of the second part of \cref{thm:theorem_haven} does not hold for \dtwNW{}-, \dtwNCW{}- and \tdtwSCD.

\begin{corollary} \label{cor:D1_bramble}
	Let $D_1$ be the digraph depicted in \cref{fig:D1}. Then, $D_1$ has no bramble of order $5$.
\end{corollary}
\begin{proof}
	By \cref{cor:D1_haven} and \cref{lem:lemma_bramble_haven}.
\end{proof}

The above corollary shows that the converse of the first part of \cref{cor:bramble_dtw} does not hold for \dtwNW{}-, \dtwNCW{}- and \tdtwSCD.

\begin{corollary} \label{cor:D1_k_linked_set}
	Let $D_1$ be the digraph depicted in~\cref{fig:D1}. Then, $D_1$ does not contain a $4$-linked set. Furthermore, it holds that $\dtwNW{D_1} \geq 4$, $\dtwNCW{D_1} \geq 4$ and $\dtwSCD{D_1} \geq 4$.
\end{corollary}
\begin{proof}
	By \cref{cor:D1_bramble} and \cref{lem:lemma_bramble_k_linked_set}.
\end{proof}

The above corollary shows that the converse of \cref{lem:lemma_balanced_w_separator} does not hold for \dtwNW{}-, \dtwNCW{}- and \tdtwSCD.
In the case of \tdtwNW, the above results are shown in \cite[Theorem 10.]{adler_directed_2007}.
Note that due to the relation shown in \cref{fig:figure_dtw_compare} and the \dtdSCE in~\cref{fig:SCE_dtd_D1}, we have $\dtwNCWE{D_1} \leq 3$.

\subsection{A Counterexample to the closure of \texorpdfstring{\dtdNCWs}{NCW-directed tree decompositions}}
\label{sec:NCW_not_closed}

The following \namecref{thm:NCW_not_closed} states that the \tdtwNCW can be larger for a butterfly minor of a digraph than for the digraph itself. 

\begin{sidewaysfigure}
	\centering		
	\input{fig4}
\end{sidewaysfigure}

\NOTClosedNCW*
\begin{proof}
    Let $D_1$, $D_1'$ be digraphs depicted in \cref{fig:D1,fig:D1'}.
    Indeed we have that $D_1 \preccurlyeq_b D_1'$, but $\dtwNCW{D_1} \nleq \dtwNCW{D_1'}$.
	The \dtdNCW\ in \cref{fig:dtd2_D1'} shows that $\dtwNCW{D_1'}$ $\leq 3$.
    However, due to~\cref{lem:D_1_NCWgeq4}, we have $\dtwNCW{D_1} \geq 4$.
\end{proof}
\section{\texorpdfstring{\tdtwNCW}{NCW-directed \treewidth} is not an upper bound on \texorpdfstring{\tdtwSCE}{SC\_0-directed \treewidth}}
\label{sec:counterexample_NCW_being_upper_bound_of_SCE}

We essentially follow the proof of \cite[Theorem 10]{adler_directed_2007} to prove \cref{thm:NCW<SCE}, which states that there is a digraph $D$ satisfying $\dtwNCW{D} < \dtwSCE{D}$, i.e.~\tdtwNCW cannot be an upper bound of \tdtwSCE. We then present \cref{cor:D2_c&r_game,cor:D2_obstructions}, which show that the exact min-max theorem between \tdtwSCE and the cops-and-robber game does not hold; moreover, the exact duality with the obstructions is not possible.

We first define two variants of an \dtdSCE. Just by ignoring \ref{D3.3}, we obtain the first one, called an \emph{\dtdSCEV}, i.e.~it is an abstract directed decomposition satisfying \ref{D3.1} and \ref{D3.2}. The corresponding directed tree-width is denoted by \dtwSCEV{D}. Then the following lemma immediately follows from the definition.

\begin{lemma} \label{lem:SCEV_leq_SCE}
	Let $D$ be a digraph. Then $\dtwSCEV{D} \leq \dtwSCE{D}$.
\end{lemma}

The second variant, called a \emph{\dtdUSCE}, is obtained by ignoring \ref{D3.3} and replacing \ref{D3.2} by 
\begin{enumerate}[leftmargin=1.5cm,labelindent=16pt,label= ($\mathsf{USC_\emptyset}$2)]
	\item \label{usc.2} for all $e = (s,t) \in E(T)$, $\beta(T_t)$ is the union of vertex sets of some strong components of $D-\gamma(e)$.
\end{enumerate}
The corresponding directed tree-width is denoted by \dtwUSCE{D}.
Since \ref{D3.2} implies \ref{usc.2}, any \dtdSCEV is a \dtdUSCE of the same width, i.e.~$\dtwUSCE{D} \leq \dtwSCEV{D}$ for any digraph $D$. The following lemma states that a stronger version of the converse is true, which speaks not only of the inequality $\dtwSCEV{D} \leq \dtwUSCE{D}$ but also of the size of bags.

\begin{lemma} \label{lem:SCEV_leq_USCE}
	Let $D$ be a digraph and $\mathcal{T} \coloneqq (T, \beta, \gamma)$ be a \dtdUSCE\ of $D$ of width $k$. Then there exits an \dtdSCEV\ $\mathcal{T'} \coloneqq $ $(T', \beta', \gamma')$ of $D$ of width at most $k$. Furthermore, there is a mapping $p: V(T') \rightarrow V(T)$ such that $|\beta'(t)| \leq |\beta(p(t))|$ for all $t \in V(T')$.
\end{lemma}

\begin{proof}
	This proof is analogous to the proof in \cite{johnson_addendum_2001}, which shows that $\dtwSCE{D} \leq \dtwNW{D}$. The same construction can be used due to \ref{usc.2} and the fact that \ref{usc.2} implies $\beta(T_t) \subseteq V(D) - \gamma(e)$ for all $e = (s,t) \in V(T)$. By construction, we have $|\beta'(t)| \leq |\beta(p(t))|$ for all $t \in V(T')$, where the mapping $p$ is the natural projection mentioned in the proof.
\end{proof}

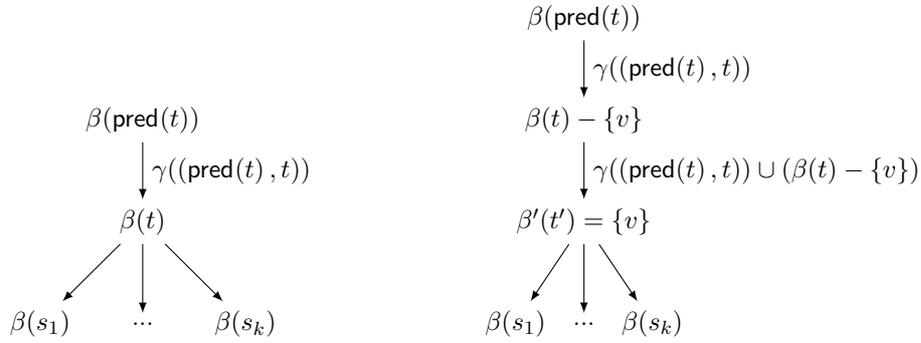
\begin{figure}[!ht]
	\centering
	\begin{subfigure}[t]{.45\linewidth}
	\centering
	\begin{tikzpicture}[level distance=1.5cm,
		level 1/.style={sibling distance=3cm},
		level 2/.style={sibling distance=1.5cm},scale=.9,transform shape]
		\tikzstyle{edge from parent}=[> = latex, ->, draw]
		\node {$\beta(\pred{t})$}
		child {node {$\beta(t)$}
				child {node {$\beta(s_1)$}}
				child {node {...}}
				child {node {$\beta(s_k)$}}
			edge from parent
			node [right] {$\gamma((\pred{t},t))$}
		};
	\end{tikzpicture}
\label{adler_lemma_a}
\end{subfigure}%
\quad
\begin{subfigure}[t]{.45\linewidth}
	\centering
	\begin{tikzpicture}[level distance=1.5cm,
		level 1/.style={sibling distance=3cm},
		level 2/.style={sibling distance=2cm},
		level 3/.style={sibling distance=1cm},scale=.9,transform shape]
		\tikzstyle{edge from parent}=[> = latex, ->, draw]
		\tikzset{label/.style={fill=white!20}}
		\node {$\beta(\pred{t})$}
		child {node {$\beta(t) - \{v\}$}
			child {node {$\beta'(t') = \{v\}$}
				child {node {$\beta(s_1)$}}
				child {node {...}}
				child {node {$\beta(s_k)$}}
				edge from parent
				node [right] {$\gamma((\pred{t},t)) \cup (\beta(t) - \{v\})$}
			}
			edge from parent
			node [right] {$\gamma((\pred{t},t))$}
		};
	\end{tikzpicture}
\label{adler_lemma_b}
\end{subfigure}%
	\caption [An auxiliary figure used in the proof of \cref{lem:Adler_lemma1}]{In the proof of \cref{lem:Adler_lemma1}, the left figure in $\mathcal{T}$ is replaced by the right one in $\mathcal{T'}$, where a new node $t'$ is added after $t$. Additionally, one of the vertices in the bag of $t$ is split off to form the new bag of $t'$.}
	\label{adler_lemma}
\end{figure}

\begin{lemma} \label{lem:Adler_lemma1}
	Let $\mathcal{T} \coloneqq $ $(T, \beta, \gamma)$ be an \dtdSCEV\ of a digraph $D$ of width $k$ with $|\beta(t)| \geq 2$ for some $t \in V(T)$. Then there exists a \dtdUSCE\ $\mathcal{T'} \coloneqq $ $(T', \beta', \gamma')$ of $D$ of width at most $k$ satisfying 
	\begin{itemize}
		\item $V(T') = V(T) \cup \{t'\}$ for a node $t' \notin V(T)$,
		\item $|\beta'(r)| = |\beta(r)|$ for all $r \in V(T) - \{t\}$,
		\item $|\beta'(t)| = |\beta(t)| - 1$, and
		\item $|\beta'(t')| = 1$.
	\end{itemize}
\end{lemma}

\begin{proof}
	We construct $\mathcal{T'} \coloneqq $ $(T', \beta', \gamma')$ as follows (see \cref{adler_lemma}):
	Let $t \in V(T)$ such that $|\beta(t)| \geq 2$, and $s_1,...,s_k \in V(T)$ be the successors of $t$ in $T$.
	\begin{itemize}
		\item $V(T') \coloneqq V(T) \cup \{t'\}$ for some $t' \notin V(T)$,
		\item $E(T') \coloneqq (E(T) - \{(t,s_i) : 1 \leq i \leq k\}) \cup \{(t',s_i) : 1 \leq i \leq k\} \cup \{(t,t')\}$,
		\item $\beta'(r) \coloneqq \beta(r)$ for all $r \in V(T) - \{t\}$,
		\item $\beta'(t) \coloneqq \beta(t) - \{v\}$ for some $v \in \beta(t)$,
		\item $\beta'(t') \coloneqq \{v\}$,
		\item $\gamma'(e) \coloneqq \gamma(e)$ for all $e \in E(T) - \{(t,s_i) : 1 \leq i \leq k\}$,
		\item $\gamma'((t',s_i)) \coloneqq \gamma((t,s_i))$ for all $1 \leq i \leq k$, and
		\item $\gamma'((t,t')) \coloneqq \gamma((\pred{t},t)) \cup (\beta(t) - \{v\})$.
	\end{itemize}
	If $t \in V(T)$ does not have a predecessor, then $\gamma((\pred{t},t))$ and $\gamma'((\pred{t},t))$ are regarded as empty sets.
	
	It directly follows from the construction that $\mathcal{T'}$ satisfies \ref{D3.1} and every requirement of the lemma. Hence, it remains to show that $\mathcal{T'}$ also satisfies \ref{usc.2}. For every $e = (s,t) \in V(T') - \{(t,t')\}$, $\beta'(T_t')$ remains the same as in $\beta(T_t)$. Since \ref{D3.2} implies \ref{usc.2}, every such edge satisfies the desired condition.
	By construction, we know $\beta'(T_t')$ is the vertex set of a strong component of $D - \gamma'((\pred{t},t))$. Therefore, $\beta'(T_{t'}') = \beta'(T_t') - (\beta(t) - \{v\})$ is the union of vertex sets of strong components of
	\begin{equation*}
		D - (\gamma'((\pred{t},t)) \cup (\beta(t) - \{v\})) = D - \gamma'((t,t')).\tag*{\qedhere}
	\end{equation*}
\end{proof}

\begin{corollary}\label{lem:Adler_lemma2}
	Let $\mathcal{T} \coloneqq $ $(T, \beta, \gamma)$ be an \dtdSCEV\ of a digraph $D$ of width $k$. Then there exists an \dtdSCEV\ $\mathcal{T'} \coloneqq (T', \beta', \gamma')$ of $D$ of width at most $k$ satisfying $|\beta'(r)| \leq 1$ for all $r \in V(T')$.
\end{corollary}

\begin{proof}
	By repeated application of \cref{lem:Adler_lemma1,lem:SCEV_leq_USCE}, we can obtain $\mathcal{T'}$ from $\mathcal{T}$.
\end{proof}

Since every \dtdSCEV\ $\mathcal{T'} \coloneqq (T', \beta', \gamma')$ provided by the above lemma has at most one vertex in each bag, if there are $s,t \in V(T')$ and $e = (s,t) \in E(T')$ such that $\beta'(t) = \{v\}$ and $\beta'(s) = \{w\}$, then $w$ is denoted by $\pred{v}$. Furthermore, if the bag of a node in $V(T')$ contains $v \in V(D)$, we simply name the node as $v$.

\begin{lemma}\label{lem:delete_emptybag_lemma_sub}
	Let $\mathcal{T} \coloneqq $ $(T, \beta, \gamma)$ be an \dtdSCEV\ of a digraph $D$ of width $k$. If $\mathcal{T}$ contains $e_1 \coloneqq (t_1,t_2) \in E(T)$ and $e_2 \coloneqq (t_2,t_3) \in E(T)$ such that
	\begin{enumerate}
		\item $\beta(t_2) = \emptyset$, and
		\item $\gamma(e_1) \subseteq \gamma(e_2)$, or $\gamma(e_2) \subseteq \gamma(e_1)$, or $\beta(T_{t_3})$ is the vertex set of a strong component of $D - (\gamma(e_1) \cap \gamma(e_2))$, 
	\end{enumerate}
	 then there is an \dtdSCEV\ $\mathcal{T'} \coloneqq (T', \beta', \gamma')$ of $D$ of width at most $k$ constructed as follows:
	\begin{itemize}
		\item $V(T') = V(T) - \{t_2\}$,
		\item $E(T') = (E(T) - \{e_1, e_2\}) \cup \{e_3 \coloneqq (t_1,t_3)\}$,
		\item $\beta'(t) = \beta(t)$ for all $t \in V(T')$,
		\item $\gamma'(e) = \gamma(e)$ for all $e \in E(T') - \{e_3\}$, and
		\item $\gamma'(e_3) = \gamma(e_1) \cap \gamma(e_2)$.
	\end{itemize}
	We call $\beta(t_2)$ a \emph{deletable empty bag}.
\end{lemma}

\begin{proof}
	Assume that $\mathcal{T}$ contains $e_1, e_2 \in E(T)$ as described above.
	As $\beta(t_2) = \emptyset$, we have $\beta(T_{t_2}) = \beta(T_{t_3})$. By construction, we know $\beta'(T_{t_3}') = \beta(T_{t_3})$. Furthermore, $\beta(T_{t_3})$ is the vertex set of a strong component of $D - \gamma(e_2)$ and also of $D - \gamma(e_1)$.
	Due to the second condition of $\mathcal{T}$, $\beta(T_{t_3}) = \beta'(T_{t_3}')$ is the vertex set of a strong component of $D - (\gamma(e_1) \cap \gamma(e_2)) = D - \gamma'(e_3)$. Then it is straightforward to check that we can obtain an \dtdSCEV\ $\mathcal{T'}$ of width at most $k$ by following the above construction.
\end{proof}

\begin{corollary}\label{lem:delete_emptybag_lemma}
	Let $\mathcal{T} \coloneqq $ $(T, \beta, \gamma)$ be an \dtdSCEV\ of a digraph $D$ of width $k$. Then there exists an \dtdSCEV\ $\mathcal{T'} \coloneqq (T', \beta', \gamma')$ of $D$ of width at most $k$, which does not contain a deletable empty bag.
\end{corollary}

\begin{proof}
	By repeated application of \cref{lem:delete_emptybag_lemma_sub}.
\end{proof}

\begin{figure}[!ht]
	\centering
\begin{tikzpicture}
	[
	state/.style={circle, draw=black, thick, minimum size=5mm, font=\small,inner sep=0,outer sep=0, minimum size=5.5mm},
	blank/.style={circle, ultra thin, minimum size=10mm},
	> = latex, 
	shorten > = 1pt, 
	auto,
	node distance = 1.5cm, 
	line width= 0.3mm
	,scale=.8,transform shape
	]
	
	\node[state] (0) {$0$};
	\node[state] (0') [below of=0] {$0'$};
	\node[state] (1) [right of=0] {$1$};
	\node[state] (1') [below of=1] {$1'$};
	\node[state] (2) [right of=1] {$2$};
	\node[state] (2') [below of=2] {$2'$};
	\node[state] (3) [right of=2] {$3$};
	\node[state] (3') [below of=3] {$3'$};
	\node[state] (4) [below right=0.35cm and 1cm of 3] {$4$};
	
	\node[state] (-1) [left of=0] {$-1$};
	\node[state] (-1') [below of=-1] {$-1'$};
	\node[state] (-2) [left of=-1] {$-2$};
	\node[state] (-2') [below of=-2] {$-2'$};
	\node[state] (-3) [left of=-2] {$-3$};
	\node[state] (-3') [below of=-3] {$-3'$};
	\node[state] (-4) [below left=0.35cm and 1cm of -3] {$-4$};

	\path[-] (0) edge node {} (0');
	\path[-] (0) edge node {} (1);
	\path[-] (0') edge node {} (1');
	\path[-] (0') edge node {} (1);
	\path[-] (1) edge node {} (1');
	\path[-] (1) edge node {} (2);
	\path[-] (1) edge node {} (1');
	\path[-] (1') edge node {} (2');
	\path[-] (1') edge node {} (2);
	\path[-] (2) edge node {} (2');
	\path[-] (2) edge node {} (3);
	\path[-] (2) edge node {} (2');
	\path[-] (2') edge node {} (3');
	\path[-] (2') edge node {} (3);
	\path[-] (3) edge node {} (3');
	\path[-] (2) edge node {} (3');
	\path[-] (3) edge node {} (4);
	\path[-] (3') edge node {} (4);
	
	\path[-] (0) edge node {} (-1);
	\path[-] (0') edge node {} (-1');
	\path[-] (0') edge node {} (-1);
	\path[-] (-1) edge node {} (-1');
	\path[-] (-1) edge node {} (-2);
	\path[-] (-1) edge node {} (-1');
	\path[-] (-1') edge node {} (-2');
	\path[-] (-1') edge node {} (-2);
	\path[-] (-2) edge node {} (-2');
	\path[-] (-2) edge node {} (-3);
	\path[-] (-2) edge node {} (-2');
	\path[-] (-2') edge node {} (-3');
	\path[-] (-2') edge node {} (-3);
	\path[-] (-3) edge node {} (-3');
	\path[-] (-2) edge node {} (-3');
	\path[-] (-3) edge node {} (-4);
	\path[-] (-3') edge node {} (-4);
	
	\path[->, red] (0) edge node {} (2');
	\path[->, red] (0') edge node {} (2);
	\path[->, red] (0) edge[bend left=30] node[above] {} (2);
	\path[->, red] (0') edge[bend right=30] node[below] {} (2');
	\path[->, red] (0) edge node {} (-2');
	\path[->, red] (0') edge node {} (-2);
	\path[->, red] (0) edge[bend right=30] node[above] {} (-2);
	\path[->, red] (0') edge[bend left=30] node[below] {} (-2');
	
	\path[->, blue] (4) edge[bend right=55] node[above] {} (0);
	\path[->, blue] (4) edge[bend left=55] node[below] {} (0');
	\path[->, blue] (-4) edge[bend left=55] node[above] {} (0);
	\path[->, blue] (-4) edge[bend right=55] node[below] {} (0');
	
\end{tikzpicture}
\captionof{figure}[The digraph $D_2$ from the proof of~\cref{thm:NCW<SCE}]{The digraph $D_2$ from the proof of~\cref{thm:NCW<SCE} with $\dtwNCW{D_2} < \dtwSCE{D_2}$.
The digraph is originally given by Adler~\cite[Fig. 2]{adler_directed_2007}.}
\label{fig:D2}
\end{figure}
\begin{figure} [!ht]
	\centering
	\begin{tikzpicture}[grow=right ,sibling distance=35pt, level distance=73pt, scale=.8,transform shape,level 1/.style={sibling distance=3cm},level 2/.style={sibling distance=1.5cm}, level 3/.style={sibling distance=1cm}]
\tikzstyle{edge from parent}=[> = latex, ->, thick, draw]
\tikzset{tree/.style={draw,rounded corners,thick, fill=blue!20, minimum size=0.6cm, level distance=1.25cm,sibling distance=2cm,inner sep=2pt}}
\tikzset{rec/.append style={rectangle,draw=black, thin, fill=white, inner sep=1pt}}

\node [tree] {$0,0'$}
child {node [tree] {$-1$} 
	child {node [tree] {$-1'$}
		child {node [tree] {$-2$}
			child {node [tree] {$-2'$}
				child {node [tree] {$-3$}
					child {node [tree] {$-3'$}
						child {node [tree] {$-4$}
							edge from parent
							node[rec, below=0.2cm] {$-3,-3'$}
						}
						edge from parent
						node[rec, below=0.2cm] {$-2,-2',-3$}
					}
					edge from parent
					node[rec, below=0.2cm] {$-2,-2'$}
				}
				edge from parent
				node[rec, below=0.2cm] {$-1',-2,-4$}
			}
			edge from parent
			node[rec, below=0.2cm] {$-1,-1',-4$}
		}
		edge from parent
		node[rec, below=0.2cm] {$0',-1,-4$}
	}
	edge from parent
	node[rec, below=0.2cm] {$0,0'$}
}
child {node [tree] {$1$} 
	child {node [tree] {$1'$}
		child {node [tree] {$2$}
			child {node [tree] {$2'$}
				child {node [tree] {$3$}
					child {node [tree] {$3'$}
						child {node [tree] {$4$}
							edge from parent
							node[rec, above=0.2cm] {$3,3'$}
						}
						edge from parent
						node[rec, above=0.2cm] {$2,2',3$}
					}
					edge from parent
					node[rec, above=0.2cm] {$2,2'$}
				}
				edge from parent
				node[rec, above=0.2cm] {$1',2,4$}
			}
			edge from parent
			node[rec, above=0.2cm] {$1,1',4$}
		}
		edge from parent
		node[rec, above=0.2cm] {$0',1,4$}
	}
	edge from parent
	node[rec, above=0.2cm] {$0,0'$}
};
\end{tikzpicture}
\captionof{figure}[An \dtdNCW of $D_2$]{An \dtdNCW of $D_2$ in \cref{fig:D2} of width $3$, implying $\dtwNCW{D_2} \leq 3$.}
\label{fig:dtdNCW_D2}
\end{figure}

\CompNCWsmallerSCE*
\begin{proof}
    Let $D_2$ be the digraph depicted in \cref{fig:D2}.
    We show that $\dtwNCW{D_2} < \dtwSCE{D_2}$.
	The \dtdNCW\ in \cref{fig:dtdNCW_D2} shows that $\dtwNCW{D_2} \leq 3$.
	We want to show that $\dtwSCE{D_2} \geq 4$. By \cref{lem:SCEV_leq_SCE,lem:Adler_lemma2,lem:delete_emptybag_lemma}, it suffices to show that $D_2$ has no \dtdSCEV\ $\mathcal{T} \coloneqq $ $(T, \beta, \gamma)$ of width $3$ such that it holds $|\beta(t)| \leq 1$ for every $t \in V(T)$, and $\mathcal{T}$ does not contain a deletable empty bag. Towards a contradiction, suppose there is such an \dtdSCEV\ $\mathcal{T} \coloneqq $ $(T, \beta, \gamma)$ of $D_2$.
	
	\begin{claim}\label{claim:c5.4.1}
		$\mathcal{T}$ has at most two leaves, namely $4$ and $-4$.
	\end{claim}
	\begin{proof*}
		Let $l \in V(T)$ be a leaf of $T$ and $e = (s,l) \in E(T)$. Then by \ref{D3.2}, $\beta(l)$ is the vertex set of a strong component of $D_2 - \gamma(e)$. Since $w(\mathcal{T}) = 3$, we have $|\Gamma(l)| \leq 4$. Every $v \in V(D_2) - \{0, 4, -4\}$ has at least four neighbours connected by an undirected edge. Therefore, no such $v$ can be contained in $\beta(l)$. If $\beta(l) = \{0\}$, then $\gamma(e)$ must contain more than three vertices because there are closed walks starting from $0$ that do not contain the three neighbours $\{-1,0',1\}$ of $0$ connected by an undirected edge. Hence, $0 \notin \beta(l)$. Then we have either $\beta(l) = \{4\}$ with $\{3,3'\} \subseteq \gamma(e)$ or $\beta(l) = \{-4\}$ with $\{-3, -3'\} \subseteq \gamma(e)$, subject to $|\gamma(e)| \leq 3$, and $4 \notin \gamma(e)$ and $-4 \notin \gamma(e)$, respectively.
	\end{proof*}

	Due to symmetry we may assume that $4$ is contained in a leaf of $T$.
	
	\begin{claim} \label{claim:c5.4.2}
		$\mathcal{T}$ has at most one node $b \in V(T)$ of out-degree $2$ (we say $b$ is a \emph{branching node}, and it \emph{branches}), and every other node has out-degree at most $1$.
	\end{claim}
	
	\begin{proof*} 
		As every outgoing edge has at least one leaf and due to \cref{claim:c5.4.1}, the claim holds.
	\end{proof*}

        \begin{figure} [!ht]
		\centering
		\begin{subfigure}[t]{.2\linewidth}
	\centering
\begin{tikzpicture}[scale=.8,transform shape,level 1/.style={sibling distance=3cm},level 2/.style={sibling distance=1.5cm}, level 3/.style={sibling distance=1cm}]
	\tikzstyle{edge from parent}=[> = latex, ->, thick, draw]
	\tikzset{tree/.style={draw,rounded corners,thick, fill=blue!20, minimum size=0.6cm, level distance=1.25cm,sibling distance=2cm,inner sep=2pt}}
	\tikzset{rec/.append style={rectangle,draw=black, thin, fill=white, inner sep=1.2pt}}
	\tikzset{blank/.style={circle, ultra thin, minimum size=10mm}}
	\tikzset{treeb/.style={draw,rounded corners,thick, fill=cyan!20, minimum size=0.6cm, level distance=1.25cm,sibling distance=2cm,inner sep=2pt}}
	
	\node [treeb] {$b$}
	child {node [tree] {$4$}
		edge from parent
		node[rec,left] {$\{3,3'\} \subseteq$}
	}
	child {node [blank] {}
		edge from parent
	};
	
\end{tikzpicture}
\subcaption{}
\end{subfigure}%
\begin{subfigure}[t]{.2\linewidth}
	\centering
	\begin{tikzpicture}[scale=.8,transform shape,level 1/.style={sibling distance=3cm},level 2/.style={sibling distance=1.5cm}, level 3/.style={sibling distance=1cm}]
		\tikzstyle{edge from parent}=[> = latex, ->, thick, draw]
		\tikzset{tree/.style={draw,rounded corners,thick, fill=blue!20, minimum size=0.6cm, level distance=1.25cm,sibling distance=2cm,inner sep=2pt}}
		\tikzset{rec/.append style={rectangle,draw=black, thin, fill=white, inner sep=1pt}}
		\tikzset{blank/.style={circle, ultra thin, minimum size=10mm}}
		\tikzset{treeb/.style={draw,rounded corners,thick, fill=cyan!20, minimum size=0.6cm, level distance=1.25cm,sibling distance=2cm,inner sep=2pt}}
		
		\node [treeb] {$b$}
		child {node [tree] {$3$}
			child {node [tree] {$4$}
				edge from parent
				node[rec] {$\{3,3'\} \subseteq$}
			}
			edge from parent
			node[rec] {$2,2',3'$}
		}
		child {node [blank] {}
			edge from parent
		};
		
	\end{tikzpicture}
	\subcaption{}
\end{subfigure}%
\begin{subfigure}[t]{.2\linewidth}
	\centering
	\begin{tikzpicture}[scale=.8,transform shape,level 1/.style={sibling distance=3cm},level 2/.style={sibling distance=1.5cm}, level 3/.style={sibling distance=1cm}]
		\tikzstyle{edge from parent}=[> = latex, ->, thick, draw]
		\tikzset{tree/.style={draw,rounded corners,thick, fill=blue!20, minimum size=0.6cm, level distance=1.25cm,sibling distance=2cm,inner sep=2pt}}
		\tikzset{rec/.append style={rectangle,draw=black, thin, fill=white, inner sep=1pt}}
		\tikzset{blank/.style={circle, ultra thin, minimum size=10mm}}
		\tikzset{treeb/.style={draw,rounded corners,thick, fill=cyan!20, minimum size=0.6cm, level distance=1.25cm,sibling distance=2cm,inner sep=2pt}}
		
		\node [treeb] {$b$}
		child {node [tree] {$3'$}
			child {node [tree] {$4$}
				edge from parent
				node[rec] {$\{3,3'\} \subseteq$}
			}
			edge from parent
			node[rec] {$2,2',3$}
		}
		child {node [blank] {}
			edge from parent
		};
		
	\end{tikzpicture}
	\subcaption{}
\end{subfigure}%
\begin{subfigure}[t]{.2\linewidth}
	\centering
	\begin{tikzpicture}[scale=.8,transform shape,level 1/.style={sibling distance=3cm},level 2/.style={sibling distance=1.5cm}, level 3/.style={sibling distance=1cm}]
		\tikzstyle{edge from parent}=[> = latex, ->, thick, draw]
		\tikzset{tree/.style={draw,rounded corners,thick, fill=blue!20, minimum size=0.6cm, level distance=1.25cm,sibling distance=2cm,inner sep=2pt}}
		\tikzset{rec/.append style={rectangle,draw=black, thin, fill=white, inner sep=1pt}}
		\tikzset{blank/.style={circle, ultra thin, minimum size=10mm}}
		\tikzset{treeb/.style={draw,rounded corners,thick, fill=cyan!20, minimum size=0.6cm, level distance=1.25cm,sibling distance=2cm,inner sep=2pt}}
		
		\node [treeb] {$b$}
		child {node [tree] {$3'$}
			child {node [tree] {$3$}
				child {node [tree] {$4$}
					edge from parent
					node[rec] {$\{3,3'\} \subseteq$}
				}
				edge from parent
				node[rec] {$2,2',3'$}
			}
			edge from parent
			node[rec,left] {$\{2,2'\} \subseteq$}
		}
		child {node [blank] {}
			edge from parent
		};
		
	\end{tikzpicture}
	\subcaption{}
\end{subfigure}%
\begin{subfigure}[t]{.2\linewidth}
	\centering
	\begin{tikzpicture}[scale=.8,transform shape,level 1/.style={sibling distance=3cm},level 2/.style={sibling distance=1.5cm}, level 3/.style={sibling distance=1cm}]
		\tikzstyle{edge from parent}=[> = latex, ->, thick, draw]
		\tikzset{tree/.style={draw,rounded corners,thick, fill=blue!20, minimum size=0.6cm, level distance=1.25cm,sibling distance=2cm,inner sep=2pt}}
		\tikzset{rec/.append style={rectangle,draw=black, thin, fill=white, inner sep=1pt}}
		\tikzset{blank/.style={circle, ultra thin, minimum size=10mm}}
		\tikzset{treeb/.style={draw,rounded corners,thick, fill=cyan!20, minimum size=0.6cm, level distance=1.25cm,sibling distance=2cm,inner sep=2pt}}
		
		\node [treeb] {$b$}
		child {node [tree] {$3$}
			child {node [tree] {$3'$}
				child {node [tree] {$4$}
					edge from parent
					node[rec] {$\{3,3'\} \subseteq$}
				}
				edge from parent
				node[rec] {$2,2',3$}
			}
			edge from parent
			node[rec,left] {$\{2,2'\} \subseteq$}
		}
		child {node [blank] {}
			edge from parent
		};
		
	\end{tikzpicture}
	\subcaption{}
\end{subfigure}%
\caption [An auxiliary figure used in the proof of \cref{thm:NCW<SCE}]{An \dtdSCEV of $D_2$ in \cref{fig:D2} of width $3$ satisfying the assumption of \cref{thm:NCW<SCE} contains one of the configurations (a)-(e).}
\label{fig:dtd3_configurations}
	\end{figure}
	
	\begin{claim}\label{claim:c5.4.3}
		W.l.o.g., assume that $4$ is a leaf of $T$. Then $\mathcal{T}$ contains one of the configurations (a)-(e) depicted in \cref{fig:dtd3_configurations}.
	\end{claim}
	
	\begin{proof*}
		By the proof of \cref{claim:c5.4.1}, we have $\{3,3'\} \subseteq \gamma((\pred{4},4))$ with $|\gamma((\pred{4},4))| \leq 3$ and $4 \notin \gamma((\pred{4},4))$.
		If $\pred{4}$ branches, we are in the case (a). Otherwise, $\{3, 3'\}$ is an inclusion-wise minimal set for the guard of the leaf $4$, and every other candidate for the guard must contain $\{3, 3'\}$ because $3, 3'$ are the neighbours of $4$ connected by an undirected edge. Moreover, $\{4\}$ is the vertex set of a strong component of $D_2 - \{3,3'\}$. Hence, if the predecessor of the leaf is empty, then it is a deletable empty bag. Then by the assumption, the predecessor of the leaf is not empty.
		By a similar argument used in the proof of \cref{claim:c5.4.1}, we have either $\pred{4} = 3$ with $\gamma((\pred{3},3)) = \{2,2',3'\}$ or $\pred{4} = 3'$ with $\gamma((\pred{3'},3')) = \{2,2',3\}$. If $\pred{\pred{4}}$ branches, then we are in the case (b) or (c). 
		
		Otherwise, by a similar argument as above we have either $\pred{3} = 3'$ with $\{2,2'\} \subseteq \gamma((\pred{3'},3'))$, $|\gamma((\pred{3'},3'))| \leq 3$ and $\{3,3',4\} \cap \gamma((\pred{3'},3')) = \emptyset$, or $\pred{3'} = 3$ with $\{2,2'\} \subseteq \gamma((\pred{3},3))$, $|\gamma((\pred{3},3))| \leq 3$ and $\{3,3',4\} \cap \gamma((\pred{3},3)) = \emptyset$.
		If $\pred{\pred{\pred{4}}}$ branches, we are in the case (d) or (e).
		Otherwise, let $v = \pred{\pred{\pred{4}}}$, and suppose $v$ does not branch. 
		By a similar argument as above $v$ is either $2$ or $2'$. Then we have $|\Gamma(v)| \geq 5$ since we need either $\{0,0',1,1',2'\} \subseteq \gamma((\pred{2},2))$ to guard $\beta(T_2) = \{2, 3, 3', 4\}$ or $\{0,0',1',2\} \subseteq \gamma((\pred{2'},2'))$ to guard $\beta(T_{2'}) = \{2', 3, 3', 4\}$. Hence, $v$ branches, and $\mathcal{T}$ contains one of the configurations (a)-(e).
	\end{proof*}
	
	\begin{claim}\label{claim:c5.4.4}
		$\mathcal{T}$ contains a branching node and has precisely two leaves, $4$ and $-4$.
	\end{claim}
	\begin{proof*}
		By \cref{claim:c5.4.1,claim:c5.4.2,claim:c5.4.3}.
	\end{proof*}

	\begin{claim}\label{claim:c5.4.5}
		$\mathcal{T}$ contains one of the configurations in \cref{fig:dtd3_configurations} for the leaf $4$ and another of their negative counterparts for the leaf $-4$ with the same $b$.
	\end{claim}
	\begin{proof*}
		By \cref{claim:c5.4.3,claim:c5.4.4} and symmetry.
	\end{proof*}
	
	Due to \cref{claim:c5.4.2,claim:c5.4.5}, we have $\{-2,-2',-1,-1',0,0',1,1',2,2'\} \subseteq \beta(b)$ and $|\Gamma(b)| > 4$, a contradiction.
\end{proof}

The above proof essentially shows that it is a strong requirement that subtrees must be disjoint from their guards in a directed tree decomposition, i.e.~$\beta(T_t) \cap \gamma(e) = \emptyset$ for all $e = (s,t) \in V(T)$ in a directed tree decomposition $\mathcal{T} \coloneqq $ $(T, \beta, \gamma)$. 
By the proof of \cref{thm:c&r_game}, the \dtdNCW\ in \cref{fig:dtdNCW_D2} yields a winning strategy for $4$ cops on $D_2$. It is noteworthy that the winning strategy is not robber-monotone.

\begin{corollary} \label{cor:D2_c&r_game}
		Let $D_2$ be the digraph depicted in \cref{fig:D2}. Then $4$ cops have a winning strategy in the cops-and-robber game on $D_2$. Furthermore, it holds that $\dtwSCE{D_2} \geq 4$.
\end{corollary}
\begin{proof}
	Due to the first part of \cref{thm:c&r_game} and the \dtdNCW\ in \cref{fig:dtdNCW_D2}, $4$ cops have a winning strategy on $D_2$. By the proof of \cref{thm:NCW<SCE}, we have $\dtwSCE{D_2} \geq 4$.
\end{proof}
The above corollary shows that the converse of the first part of \cref{thm:c&r_game} does not hold for \tdtwSCE.

\begin{corollary} \label{cor:D2_obstructions}
	Let $D_2$ be the digraph depicted in \cref{fig:D2}. Then, $D_2$ has no haven of order $5$, no bramble of order $5$, and no $4$-linked set.
\end{corollary}
\begin{proof}
	Due to \cref{cor:D1_c&r_game}, the first part of \cref{thm:theorem_haven}, \cref{lem:lemma_bramble_haven,lem:lemma_bramble_k_linked_set}.
\end{proof}
The above corollary shows that the following does not hold for \tdtwSCE : the converse of the second part of \cref{thm:theorem_haven}, the converse of the first part of \cref{cor:bramble_dtw}, the converse of \cref{lem:lemma_balanced_w_separator}.
The result of \cref{cor:D1_c&r_game,cor:D2_c&r_game} indicate that the \tdtwSCE\ of a digraph $D$, along with \dtwNW{}-, \dtwNCW{}- and \tdtwSCD , is not equal to the minimal number of cops needed to win minus one in the cops-and-robber game on $D$.
Note that due to the relation shown in \cref{fig:figure_dtw_compare} and the \dtdNCW\ in \cref{fig:dtdNCW_D2}, we have $\dtwNCWE{D_2} \leq 3$.

\subsection{A counterexample to the closure of robber-monotone winning strategies}
\label{sec:c&r_of_B_minor}

In this \namecref{sec:c&r_of_B_minor}, we consider another question to which the digraph $D_2$, see \cref{fig:D2}, yields a counterexample.
It shows that while the cops-and-robber games are closed under taking butterfly minors (see \cite{ganian_are_2016} for similar work), this is not the case if the cops have to play in a robber-monotone way.

\begin{observation} \label{obs:c&r_game_closed_under_b_minor}
	Let $D$, $D'$ be digraphs such that $D' \preccurlyeq_b D$.
	If $k$ cops have a winning strategy on $D$, then $k$ cops have a winning strategy on $D'$.
\end{observation}

The intuition behind the \namecref{obs:c&r_game_closed_under_b_minor} is that deleting vertices and edges or shrinking induced paths by butterfly contracting edges does not help the robber elude cops.
If $D'$ is obtained from $D$ by deleting some vertices and edges, then the cops' winning strategy on $D$ can be used on $D'$ to win, where the cops occupy the vertices that remain in $D'$. 
Let us assume that $D'$ is obtained from $D$ by butterfly contracting $e = (s,t) \in E(D)$ into the vertex $x \in V(D')$.
If there is a closed walk $W$ that passes $s$ or $t$ or both in $D$, then there is also a closed walk in $D'$ that passes $x$ instead of $s$ or $t$ or both and passes the same vertices of $W$ in the same order. Furthermore, the converse is also true by \cref{lem:walkInMinor}.
Therefore, if a cop has to occupy $s$ or $t$ (or possibly two cops have to occupy both) at some point in the winning strategy on $D$, then a cop can occupy $x$ in $D'$ instead.

Regarding the above \namecref{obs:c&r_game_closed_under_b_minor}, one might ask whether the number of cops needed to win the game in a robber-monotone way is closed under taking butterfly minors. The following counterexample shows that the answer is negative.

\begin{figure}[!ht]
\begin{center}
\begin{tikzpicture}
	[
	state/.style={circle, draw=black, thick, minimum size=5mm, font=\small,inner sep=0,outer sep=0, minimum size=5.5mm},
	blank/.style={circle, ultra thin, minimum size=10mm},
	> = latex, 
	shorten > = 1pt, 
	auto,
	node distance = 1.5cm, 
	line width= 0.3mm
	,scale=.8,transform shape
	]

	\node[state] (0) {$0$};
	\node[state] (0') [below of=0] {$0'$};
	\node[state] (1) [right of=0] {$1$};
	\node[state] (1') [below of=1] {$1'$};
	\node[state] (2) [right of=1] {$2$};
	\node[state] (2') [below of=2] {$2'$};
	\node[state] (3) [right of=2] {$3$};
	\node[state] (3') [below of=3] {$3'$};
	\node[state] (4) [below right=0.35cm and 1cm of 3] {$4$};
	\node[state] (5) [right of=4] {$5$};
	
	\node[state] (-1) [left of=0] {$-1$};
	\node[state] (-1') [below of=-1] {$-1'$};
	\node[state] (-2) [left of=-1] {$-2$};
	\node[state] (-2') [below of=-2] {$-2'$};
	\node[state] (-3) [left of=-2] {$-3$};
	\node[state] (-3') [below of=-3] {$-3'$};
	\node[state] (-4) [below left=0.35cm and 1cm of -3] {$-4$};
	\node[state] (-5) [left of=-4] {$-5$};

	\path[-] (0) edge node {} (0');
	\path[-] (0) edge node {} (1);
	\path[-] (0') edge node {} (1');
	\path[-] (0') edge node {} (1);
	\path[-] (1) edge node {} (1');
	\path[-] (1) edge node {} (2);
	\path[-] (1) edge node {} (1');
	\path[-] (1') edge node {} (2');
	\path[-] (1') edge node {} (2);
	\path[-] (2) edge node {} (2');
	\path[-] (2) edge node {} (3);
	\path[-] (2) edge node {} (2');
	\path[-] (2') edge node {} (3');
	\path[-] (2') edge node {} (3);
	\path[-] (3) edge node {} (3');
	\path[-] (2) edge node {} (3');
	\path[-] (3) edge node {} (4);
	\path[-] (3') edge node {} (4);
	\path[->, red] (4) edge node {} (5);
	
	\path[-] (0) edge node {} (-1);
	\path[-] (0') edge node {} (-1');
	\path[-] (0') edge node {} (-1);
	\path[-] (-1) edge node {} (-1');
	\path[-] (-1) edge node {} (-2);
	\path[-] (-1) edge node {} (-1');
	\path[-] (-1') edge node {} (-2');
	\path[-] (-1') edge node {} (-2);
	\path[-] (-2) edge node {} (-2');
	\path[-] (-2) edge node {} (-3);
	\path[-] (-2) edge node {} (-2');
	\path[-] (-2') edge node {} (-3');
	\path[-] (-2') edge node {} (-3);
	\path[-] (-3) edge node {} (-3');
	\path[-] (-2) edge node {} (-3');
	\path[-] (-3) edge node {} (-4);
	\path[-] (-3') edge node {} (-4);
	\path[->, red] (-4) edge node {} (-5);
	
	\path[->, red] (0) edge node {} (2');
	\path[->, red] (0') edge node {} (2);
	\path[->, red] (0) edge[bend left=30] node[above] {} (2);
	\path[->, red] (0') edge[bend right=30] node[below] {} (2');
	\path[->, red] (0) edge node {} (-2');
	\path[->, red] (0') edge node {} (-2);
	\path[->, red] (0) edge[bend right=30] node[above] {} (-2);
	\path[->, red] (0') edge[bend left=30] node[below] {} (-2');
	
	\path[->, blue] (5) edge[bend right=50] node[above] {} (0);
	\path[->, blue] (5) edge[bend left=50] node[below] {} (0');
	\path[->, blue] (-5) edge[bend left=50] node[above] {} (0);
	\path[->, blue] (-5) edge[bend right=50] node[below] {} (0');
	
\end{tikzpicture}
\captionof{figure}[The digraph $D_2'$ from \cref{thm:c&r_game_closure}, $D_2$ is a butterfly minor of $D_2'$]{The digraph $D_2'$ from \cref{thm:c&r_game_closure}, a modification of \cite[Fig. 2]{adler_directed_2007}. The digraph $D_2$ in \cref{fig:D2} is a butterfly minor of $D_2'$.}
\label{fig:D2'}
\end{center}
\end{figure}

\ThmCopsAndRobbermonotonenotclosed*

\begin{proof}
    Let $D_2$, $D_2'$ again be digraphs depicted in \cref{fig:D2,fig:D2'}. Then $D_2 \preccurlyeq_b D_2'$.
	The following strategy is a robber-monotone winning strategy for $4$ cops on $D_2'$. The first position is $\{0,0',1,-1\}$. Due to symmetry, we may assume that the robber is in the positive part. Then the cops move to $\{0,0',1,5\}$ then to $\{0',1,1',5\}$, $\{1,1',2,5\}$, $\{1',2,2',5\}$, $\{2,2',3,5\}$, $\{2,2',3,3'\}$, $\{3,3',4\}$.
	We refer to \cite[Theorem 8]{adler_directed_2007} for the proof that the robber can win against $4$ cops following a robber-monotone strategy on~$D_2$.
\end{proof}

Note that $4$ cops have a (non-robber-monotone) winning strategy on $D_2$, so this does not contradict \cref{obs:c&r_game_closed_under_b_minor}.
The winning strategy is as follows: the first position for the cops is $\{0,0',1,-1\}$. Due to symmetry, we may assume that the robber is in the positive part. Then the cops move to $\{0,0',1,4\}$ then to $\{0',1,1',4\}$, $\{1,1',2,4\}$, $\{1',2,2',4\}$, $\{2,2',3,4\}$, $\{2,2',3,3'\}$, $\{3,3',4\}$. When cops switch their position from $\{2,2',3,4\}$ to $\{2,2',3,3'\}$, the robber can move from $3'$ to $4$, which was not included in the robber space before.

\subsection{A counterexample to the closure of \texorpdfstring{$\dtwSCE{}$}{SC\_0}- and \texorpdfstring{\tdtwSCD}{SC\_d-directed \treewidth}}
\label{sec:SC0andSCDnotClosed}

This section shows that the two notions \tdtwSCE and \tdtwSCD are not closed under taking butterfly minors.

\ThmSCENotClosed*
\begin{proof}
    This is witnessed by the digraphs $D_2$ and $D_2'$ depicted in \cref{fig:D2,fig:D2'} as $D_2 \preccurlyeq_b D_2'$, but $\dtwSCE{D_2} \nleq \dtwSCE{D_2'}$.
	Indeed, there is an \dtdSCE\, shown in \cref{fig:dtd3_D2'}, which proves that $\dtwSCE{D_2'}$ $\leq 3$.
    However, we have $\dtwSCE{D_2} \geq 4$, due to \cref{cor:D2_c&r_game}.
\end{proof}

\begin{figure}[!ht]
\begin{center}
	\begin{tikzpicture}[grow=right ,sibling distance=35pt, level distance=73pt, scale=.8,transform shape,level 1/.style={sibling distance=3cm},level 2/.style={sibling distance=1.5cm}, level 3/.style={sibling distance=1cm}]
	\tikzstyle{edge from parent}=[> = latex, ->, thick, draw]
	\tikzset{tree/.style={draw,rounded corners,thick, fill=blue!20, minimum size=0.6cm, level distance=1.25cm,sibling distance=2cm,inner sep=2pt}}
	\tikzset{rec/.append style={rectangle,draw=black, thin, fill=white, inner sep=1pt}}
	
	\node [tree] {$0,0'$}
	child {node [tree] {$-5$}
		edge from parent
		node[rec, below=0.3cm] {$0,0'$}
	}
	child {node [tree] {$-1$} 
		child {node [tree] {$-1'$}
			child {node [tree] {$-2$}
				child {node [tree] {$-2'$}
					child {node [tree] {$-3$}
						child {node [tree] {$-3'$}
							child {node [tree] {$-4$}
								edge from parent
								node[rec, below=0.3cm] {$-3,-3'$}
							}
							edge from parent
							node[rec, below=0.3cm] {$-2,-2',-3$}
						}
						edge from parent
						node[rec, below=0.3cm] {$-2,-2',-5$}
					}
					edge from parent
					node[rec, below=0.3cm] {$-1',-2,-5$}
				}
				edge from parent
				node[rec, below=0.3cm] {$-1,-1',-5$}
			}
			edge from parent
			node[rec, below=0.3cm] {$0',-1,-5$}
		}
		edge from parent
		node[rec] {$0,0'$}
	}
	child {node [tree] {$1$} 
		child {node [tree] {$1'$}
			child {node [tree] {$2$}
				child {node [tree] {$2'$}
					child {node [tree] {$3$}
						child {node [tree] {$3'$}
							child {node [tree] {$4$}
								edge from parent
								node[rec, above=0.2cm] {$3,3'$}
							}
							edge from parent
							node[rec, above=0.2cm] {$2,2',3$}
						}
						edge from parent
						node[rec, above=0.2cm] {$2,2',5$}
					}
					edge from parent
					node[rec, above=0.2cm] {$1',2,5$}
				}
				edge from parent
				node[rec, above=0.2cm] {$1,1',5$}
			}
			edge from parent
			node[rec, above=0.2cm] {$0',1,5$}
		}
		edge from parent
		node[rec] {$0,0'$}
	}
	child {node [tree] {$5$}
		edge from parent
		node[rec, above=0.2cm] {$0,0'$}
	};
	
\end{tikzpicture}
\captionof{figure}[An $\dtwSCE{}$- or an \dtdSCD of $D_2'$]{An $\dtwSCE{}$- or an \dtdSCD of $D_2'$ in \cref{fig:D2'} of width $3$, implying $\dtwSCE{D_2'}$ $\leq 3$ and $\dtwSCD{D_2'}$ $\leq 3$.}
\label{fig:dtd3_D2'}
\end{center}
\end{figure}

\ThmSCDNotClosed*
\begin{proof}
    Again, $D_2$ and $D_2'$ from \cref{fig:D2,fig:D2'} yield the counterexample as we have that $D_2 \preccurlyeq_b D_2'$, but $\dtwSCD{D_2} \nleq \dtwSCD{D_2'}$.
	The \dtdSCD\ in \cref{fig:dtd3_D2'} shows that $\dtwSCD{D_2'} \leq 3$.
    Since $\dtwSCE{D} \leq \dtwSCD{D}$ holds for every digraph $D$ (see~\cref{fig:figure_dtw_compare}) and due to \cref{cor:D2_c&r_game}, we have $\dtwSCE{D_2} \geq 4$, and thus $\dtwSCD{D_2} \geq 4$.
\end{proof}

\newpage
\bibliographystyle{alphaurl}
\bibliography{references}

\end{document}